\documentclass[12pt]{amsart}

  \usepackage{graphicx}
  \usepackage{pictexwd,dcpic}
  \usepackage{latexsym,amssymb,epsfig,a4wide}

  \def\KK{{\mathbb K}}
  
  \def\NN{{\mathbb N}}

  \def\RR{{\mathbb R}}

  \def\fF{{\mathcal F}}

  \def\inte{{\rm int}}
  \def\Inte{{\rm Int}}
  \def\link{{\rm lk}}
  \def\KR{{\rm K}}
  \def\LR{{\rm L}}
  \def\min{{\rm min}}
  
  \def\vert{{\rm vert}}
  \def\sm{\smallsetminus}

  \theoremstyle{plain}
    \newtheorem{theorem}{Theorem}[section]
    \newtheorem{proposition}[theorem]{Proposition}
    
    \newtheorem{corollary}[theorem]{Corollary}
    
  \theoremstyle{definition}
    \newtheorem{definition}[theorem]{Definition}
    \newtheorem{example}[theorem]{Example}
    
    \newtheorem{question}[theorem]{Question}
    \newtheorem{remark}[theorem]{Remark}

  \numberwithin{equation}{section}

\begin{document}

\title[Cubical subdivisions and local $h$-vectors]{Cubical subdivisions
and local $h$-vectors}

  \author{Christos~A.~Athanasiadis}
  \address{Department of Mathematics
          (Division of Algebra-Geometry) \\
          University of Athens \\
          Panepistimioupolis, Athens 15784 \\
          Hellas (Greece)}
  \email{caath@math.uoa.gr}

\date{August 25, 2010; Revised, January 28, 2011}

  \begin{abstract}
    Face numbers of triangulations of simplicial complexes were studied by Stanley
    by use of his concept of a local $h$-vector. It is shown that a parallel theory
    exists for cubical subdivisions of cubical complexes, in which the role of the
    $h$-vector of a simplicial complex is played by the (short or long) cubical
    $h$-vector of a cubical complex, defined by Adin, and the role of the local
    $h$-vector of a triangulation of a simplex is played by the (short or long)
    cubical local $h$-vector of a cubical subdivision of a cube. The cubical local
    $h$-vectors are defined in this paper and are shown to share many of the
    properties of their simplicial counterparts. Generalizations to subdivisions of
    locally Eulerian posets are also discussed.
  \end{abstract}

  \maketitle

  \section{Introduction}
  \label{sec:intro}

  Simplicial subdivisions (or triangulations) of simplicial complexes were studied
  from an enumerative point of view by Stanley \cite{Sta92}. Specifically, the paper
  \cite{Sta92} is concerned with the way in which the face enumeration of a simplicial
  complex $\Delta$, presented in the form of the $h$-vector (equivalently, of the
  $h$-polynomial) of $\Delta$, changes under simplicial subdivisions of various types;
  see \cite[Chapter II]{StaCCA} for the importance of $h$-vectors in the combinatorics
  of simplicial complexes.

  A key result in this study is a formula \cite[Equation (2)]{Sta92} which expresses
  the $h$-polynomial of a simplicial subdivision $\Delta'$ of a pure simplicial
  complex $\Delta$ as the sum of the $h$-polynomial of $\Delta$ and other terms, one
  for each nonempty face $F$ of $\Delta$. The term corresponding to $F$ is a product
  of two polynomials, one of which depends only on the local combinatorics of
  $\Delta$ at $F$ and the other only on the restriction of $\Delta'$ to $F$. The
  latter is determined by the local $h$-vector of $\Delta'$ at $F$, a concept which
  is introduced and studied in \cite{Sta92}. It is shown that the local $h$-vector
  of a simplicial subdivision of a simplex is symmetric and that for quasi-geometric
  subdivisions (a kind of topological subdivision which includes all geometric
  subdivisions) it is nonnegative. From these results it is deduced that the $h$-vector
  of a Cohen-Macaulay simplicial complex increases under quasi-geometric subdivision.
  Other properties of local $h$-vectors are also established in \cite{Sta92} and a
  generalization of the theory to formal subdivisions of lower Eulerian posets is
  developed.

  This paper investigates similar questions for cubical subdivisions of cubical
  complexes. A well-behaved analogue of the $h$-polynomial was introduced for cubical
  complexes (in two forms, short and long) by Adin \cite{Ad96}. For both types,
  it is shown that formulas analogous to that of \cite{Sta92} hold for the cubical
  $h$-polynomial of a cubical subdivision of a pure cubical complex (Theorems
  \ref{thm:local} and \ref{thm:clocal}), when one suitably defines the (short or
  long) cubical local $h$-vector of a cubical subdivision of a cube (Definitions
  \ref{def:short} and \ref{def:long}). The cubical local $h$-vectors are shown
  to have symmetric coefficients (Theorem \ref{thm:symm} and Corollary
  \ref{cor:symm}). For the short type, they are shown to have nonnegative
  coefficients for a class of subdivisions which are called locally
  quasi-geometric in this paper and which include all geometric cubical
  subdivisions. A monotonicity property, analogous to that of \cite{Sta92} in
  the simplicial case, is deduced for the short cubical $h$-vectors of locally
  quasi-geometric cubical subdivisions (Corollary \ref{cor:monotone}). At present,
  nonnegativity of cubical local $h$-vectors and monotonicity of cubical
  $h$-vectors for (say, geometric) cubical subdivisions seem to be out of reach
  for the long type in high dimensions.

  The theory of short cubical local $h$-vectors is generalized to formal
  subdivisions of locally Eulerian posets in Section \ref{sec:gen} by suitably
  modifying the approach of \cite[Part II]{Sta92}. This includes other kinds of
  $h$-vectors for simplicial subdivisions, such as the short simplicial $h$-vector
  of Hersh and Novik \cite{HN02} (see Examples \ref{ex:cchgen} and \ref{ex:newsub}),
  as well as cubical subdivisions of complexes more general than regular
  CW-complexes (see Section \ref{subsec:CW}), to which the theory of \cite{Sta92}
  can be extended.

  This paper is organized as follows. Preliminaries on (simplicial and) cubical
  complexes are given in Section \ref{sec:back}. Short and long cubical local
  $h$-vectors are defined in Sections \ref{sec:short} and \ref{sec:long},
  respectively, where several examples and elementary properties also appear.
  The main properties of short cubical local $h$-vectors are stated and proven in
  Sections \ref{sec:main} and \ref{sec:nonn}. Section \ref{sec:long} extends some
  of these properties to the long cubical local $h$-vector by observing that the two
  types of cubical local $h$-vectors are related in a simple way. Section
  \ref{sec:gen} is devoted to generalizations of the theory of short cubical local
  $h$-vectors to subdivisions of locally Eulerian posets. The methods used in this
  paper are similar to those of \cite{Sta92}, but extensions and variations are
  occasionally needed.

  \section{Face enumeration and subdivisions of cubical complexes}
  \label{sec:back}

  This section reviews background material on the face enumeration and subdivisions
  of simplicial and cubical complexes. We refer the reader to \cite{StaCCA} for any
  undefined terminology and for more information on the algebraic, enumerative and
  homological properties of simplicial complexes. Basic background on partially
  ordered sets and convex polytopes can be found in \cite[Chapter 3]{StaEC1} and
  \cite{Gru67, Zie95}, respectively. We denote by $|S|$ the cardinality of a finite
  set $S$.

  \subsection{Simplicial complexes} \label{subsec:sim}
  Let $\Delta$ be a (finite, abstract) simplicial complex of dimension $d-1$. The
  fundamental enumerative invariant of $\Delta$ for us will be the $h$-polynomial,
  defined by
    \begin{equation} \label{eq:hsimplicial}
      h(\Delta, x) \ = \ \sum_{F \in \Delta} \ x^{|F|} (1-x)^{d-|F|}.
    \end{equation}
  This polynomial has nonnegative coefficients if $\Delta$ is Cohen-Macaulay over
  some field. If $\Delta$ is a homology ball over some field, we let $\inte(\Delta)$ 
  denote the set of (necessarily nonempty) faces of $\Delta$ which are not contained 
  in the boundary $\partial \Delta$ of $\Delta$ (observe that $\inte(\Delta)$ is 
  \emph{not} a subcomplex of $\Delta$; for instance, if $\Delta$ is the simplex $2^V$
  on the vertex set $V$, then $\inte(\Delta) = \{V\}$). If $\Delta$ is a homology 
  sphere, we set $\inte(\Delta) := \Delta$ (for a discussion of homology balls and 
  spheres see, for instance, \cite[Section 2]{EK07} or \cite[Section 4]{Swa06}). In 
  either case, we set
    \begin{equation} \label{eq:hinte}
      h(\inte(\Delta), x) \ = \ \sum_{F \in \inte(\Delta)} \ x^{|F|}
      (1-x)^{d-|F|}.
    \end{equation}
  The next proposition follows from Theorem 7.1 in \cite[Chapter II]{StaCCA} (see
  also \cite[Theorem 2]{MW71}, \cite[Lemma 6.2]{Sta87} and \cite[Lemma 2.3]{Sta93}).

  \begin{proposition} \label{prop:sirec}
    Let $\Delta$ be a $(d-1)$-dimensional simplicial complex. If $\Delta$ is either
    a homology ball or a homology sphere over some field, then
      \begin{equation} \label{eq:propsirec}
        x^d \, h (\Delta, 1/x) \ = \ h(\inte(\Delta), x).
      \end{equation}
  \end{proposition}

  \subsection{Cubical complexes} \label{subsec:cub}
  An (abstract) $d$-dimensional cube for us will be any poset which is isomorphic to
  the poset of faces of the standard $d$-dimensional cube $[0, 1]^d \subseteq \RR^d$.
  A (finite, abstract) \emph{cubical complex} is a finite poset $\KR$ with a minimum
  element, denoted by $\varnothing$ and refered to as the empty face, having the
  following properties: (i) the interval $[\varnothing, F]$ in $\KR$ is an abstract
  cube for every $F \in \KR$; and (ii) $\KR$ is a meet-semilattice (meaning that any
  two elements of $\KR$ have a greatest lower bound). The elements of $\KR$ are called 
  \emph{faces} and have well-defined dimensions. We will refer to the meet (greatest 
  lower bound) of two faces of $\KR$ as their \emph{intersection} and will say that 
  a face $F$ of $\KR$ is \emph{contained} in a face $G$ if $F \le G$ holds in $\KR$.
  We will refer to a nonempty  order ideal of $\KR$ as a \emph{subcomplex}. The
  complex $\KR$ is \emph{pure} if all its maximal elements have the same dimension.
  Throughout this paper we will denote by $\fF(\KR)$ the subposet of nonempty faces
  of $\KR$. Any algebraic or topological properties of $\KR$ we consider will refer
  to those of the simplicial complex of chains (order complex) of $\fF(\KR)$; see
  \cite{Bj95}.

  A geometric cube is any polytope which is combinatorially equivalent to a standard
  cube. A (finite) \emph{geometric cubical complex} is a is a finite collection $K$
  of geometric cubes in some space $\RR^N$, called faces, which has the properties:
  (i) every face of an element of $K$ also belongs to $K$; and (ii) the intersection
  of any two elements of $K$ is a face of both. The set of faces of any geometric
  cubical complex, ordered by inclusion, is an (abstract) cubical complex but the
  converse is not always true; see \cite[Section 1]{Het95}. However, every cubical 
  complex $\KR$ is isomorphic to the poset of faces of a regular CW-complex (forced 
  to have the intersection property), and that complex is determined from $\KR$ up 
  to homeomorphism; see \cite{Bj84} \cite[Section 4.7]{BSZ99} and Section 
  \ref{subsec:CW}.

  The short cubical $h$-vector of a cubical complex $\KR$ is defined in \cite{Ad96}
  as
    \begin{equation} \label{eq:defsc}
      h^{(sc)} (\KR, x) \ = \ \sum_{F \in \KR \sm \{\varnothing\}} \ (2x)^{\dim(F)}
      (1-x)^{d-\dim(F)},
    \end{equation}
  where $d = \dim (\KR)$ is the maximum dimension of a face of $\KR$. The link
  $\link_\KR (F)$ of any nonempty face $F$ of $\KR$ is naturally an abstract
  simplicial complex and we have the fundamental observation \cite[Theorem
  9]{Ad96} (due to G.~Hetyei) that if $\KR$ is pure, then
    \begin{equation} \label{eq:lemsc}
      h^{(sc)} (\KR, x) \ = \ \sum_{v \in \vert(\KR)} \ h (\link_\KR (v), x),
    \end{equation}
  where $\vert(\KR)$ denotes the set of vertices (zero-dimensional faces) of $\KR$.

  Given a cubical complex $\KR$ of dimension $d$ and a cubical complex $\KR'$ of
  dimension $d'$, the direct product $\fF(\KR) \times \fF(\KR')$ is isomorphic to
  the poset $\fF(\KR \times \KR')$ of nonempty faces of a cubical complex $\KR
  \times \KR'$ of dimension $d+d'$ and we have $\dim (F \times F') = \dim(F) +
  \dim(F')$ for all $F \in \fF(\KR)$ and $F' \in \fF(\KR')$. Hence, equation
  (\ref{eq:defsc}) yields
    \begin{equation} \label{eq:prodsc}
      h^{(sc)} (\KR \times \KR', x) \ = \ h^{(sc)} (\KR, x) \, h^{(sc)} (\KR', x).
    \end{equation}

  By analogy with (\ref{eq:hinte}), we define
    \begin{equation} \label{eq:schinte}
      h^{(sc)} (\inte(\KR), x) \ = \ \sum_{F \in \inte(\KR)} \ (2x)^{\dim(F)}
      (1-x)^{d-\dim(F)}
    \end{equation}
  for every cubical complex $\KR$ which is homeomorphic to a $d$-dimensional
  manifold with nonempty boundary $\partial \KR$, where $\inte(\KR) = \KR \sm
  \partial \KR$. The next statement has appeared in \cite{Kle09}. We include
  a proof, so that it can be easily adapted in other situations (see Example
  \ref{ex:newsub}). The statement (and proof) is valid for manifolds without
  boundary as well (see part (i) of \cite[Theorem 5]{Ad96}), provided one
  sets $\inte(\KR) = \KR$. The corresponding statement for the long cubical
  $h$-vector (defined in the sequel) of a cubical ball appeared in 
  \cite[Proposition 4.1]{BBC97}.

  \begin{corollary} \label{cor:screc} {\rm (\cite[Proposition 4.4]{Kle09})}
    For every cubical complex $\KR$ which is homeomorphic to a $d$-dimensional
    manifold with nonempty boundary we have
      \begin{equation} \label{eq:propscrec}
        x^d \, h^{(sc)} (\KR, 1/x) \ = \ h^{(sc)} (\inte(\KR), x).
      \end{equation}
  \end{corollary}
  \begin{proof}
    We observe that for every vertex $v$ of $\KR$, the simplicial complex
    $\link_\KR (v)$ is either a homology sphere or a homology ball,
    depending on whether $v$ lies in $\inte(\KR)$ or not, of dimension $d-1$.
    Thus, using (\ref{eq:lemsc}) and Proposition \ref{prop:sirec}, we compute that

      \begin{eqnarray*}
      x^d \, h^{(sc)} (\KR, 1/x) &=& \sum_{v \in \vert(\KR)} \ x^d \,
      h (\link_\KR (v), 1/x) \ =
      \sum_{v \in \vert(\KR)} \ h (\inte(\link_\KR (v)), x) \\
      & & \\
      &=& \sum_{v \in \vert(\KR)} \ \sum_{E \in \inte(\link_\KR (v))}
      x^{|E|} (1-x)^{d-|E|} \\
      & & \\
      &=& \sum_{F \in \KR} \ n(F) \, x^{\dim(F)} (1-x)^{d-\dim(F)},
      \end{eqnarray*}
    where $n(F)$ denotes the number of vertices $v$ of $F$ for which the
    face of $\link_\KR (v)$ which corresponds to $F$ is an interior face
    of $\link_\KR (v)$. Equation (\ref{eq:propscrec}) follows by noting
    that
      $$ n(F) \ = \ \begin{cases}
         2^{\dim(F)}, & \text{if \ $F \in \inte(\KR)$} \\
         0, & \text{otherwise} \end{cases} $$
    for $F \in \KR$.
  \end{proof}

  The (long) cubical $h$-polynomial of a $d$-dimensional cubical complex $\KR$
  may be defined \cite{Ad96} by the equation
    \begin{equation} \label{eq:defc}
      (x+1) h^{(c)} (\KR, x) \ = \ 2^d + x h^{(sc)} (\KR, x) + (-2)^d \,
      \widetilde{\chi} (\KR) x^{d+2},
    \end{equation}
  where
    $$ \widetilde{\chi} (\KR) \ = \ \sum_{F \in \KR} (-1)^{\dim(F)} $$
  is the reduced Euler characteristic of $\KR$. Since $h^{(sc)} (\KR, -1) = 2^d
  + 2^d \, \widetilde{\chi} (\KR)$ \cite[Lemma 1 (ii)]{Ad96}, the function
  $h^{(c)} (\KR, x)$ is indeed a polynomial in $x$ of degree at most $d+1$. For
  use in Section \ref{sec:long}, we note (see also \cite[Lemma 2]{Ad96}) that
    \begin{equation} \label{eq:h0}
      h^{(c)}_0 (\KR) \ = \ 2^d
    \end{equation}
  and
    \begin{equation} \label{eq:hd+1}
      h^{(c)}_{d+1} (\KR) \ = \ (-2)^d \, \widetilde{\chi} (\KR),
    \end{equation}
  where $h^{(c)} (\KR, x) = \sum_{i=0}^{d+1} h^{(c)}_i (\KR) x^i$.

  \subsection{Cubical subdivisions} \label{subsec:cubsub}
  The notion of cubical subdivision of a cubical complex we will adopt is
  analogous to that of a simplicial subdivision of an abstract simplicial complex,
  introduced in \cite[Section 2]{Sta92}.

  A (finite, topological) \emph{cubical subdivision} of a cubical complex $\KR$
  is a cubical complex $\KR'$, together with a map $\sigma: \fF(\KR') \to \fF(\KR)$
  (the associated subdivision map), having the following properties:
    \begin{itemize}
      \item[(a)] 
      For every $F \in \fF(\KR)$, the set $\KR'_F := \sigma^{-1} (\fF(F)) \cup
      \{\varnothing\}$ is a subcomplex of $\KR'$ which is homeomorphic to a ball of
      dimension $\dim(F)$. 
      \item[(b)] 
      $\sigma^{-1} (F)$ consists of the interior faces of the ball $\KR'_F$. 
    \end{itemize}
  The subcomplex $\KR'_F$ is called the \emph{restriction} of $\KR'$ to $F$. For $G 
  \in \fF(\KR')$, the face $\sigma(G)$ of $\KR$ is called the \emph{carrier} of $G$. 
  It follows from the defining properties that $\sigma$ is surjective and that $\dim(G) 
  \le \dim (\sigma(G))$ for every $G \in \fF(\KR')$. Several examples of these concepts 
  appear in the following sections.

  Let $C$ be a geometric cube. A cubical subdivision $\Gamma$ of $C$ is said to be
  \emph{geometric} if it can be realized by a geometric cubical complex $\Gamma'$
  which subdivides $C$ (meaning, the union of faces of $\Gamma'$ is equal to $C$).
  In that case, the carrier of $G \in \Gamma' \sm\{\varnothing\}$ is the smallest
  face of $C$ which contains $G$.

  \section{The short cubical local $h$-vector}
  \label{sec:short}

  This section includes definitions, basic examples and elementary properties of
  short cubical local $h$-vectors.

  \begin{definition} \label{def:short}
    Let $C$ be a $d$-dimensional cube. For any cubical
    subdivision $\Gamma$ of $C$, we define a polynomial $\ell_C (\Gamma, x)
    = \ell_0 + \ell_1 x + \cdots + \ell_d x^d$ by
      \begin{equation} \label{eq:defshort}
        \ell_C (\Gamma, x) \ = \sum_{F \in \fF(C)} \ (-1)^{d - \dim(F)} \,
        h^{(sc)} (\Gamma_F, x).
      \end{equation}
    We call $\ell_C (\Gamma, x)$ the \emph{short cubical local $h$-polynomial}
    of $\Gamma$ (with respect to $C$). We call $\ell_C (\Gamma) = (\ell_0,
    \ell_1,\dots,\ell_d)$ the \emph{short cubical local $h$-vector} of $\Gamma$
    (with respect to $C$).
  \end{definition}

  \begin{figure}
  \epsfysize = 2.0 in \centerline{\epsffile{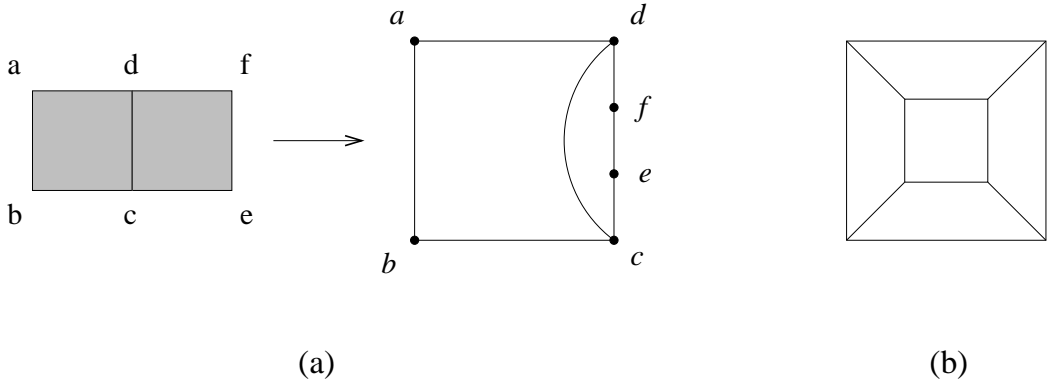}}
  \caption{Two cubical subdivisions of a square.}
  \label{fig:cubsub}
  \end{figure}

  \begin{example} \label{ex:short}
    (a) For the trivial subdivision $\Gamma = C$ of the $d$-dimensional cube $C$ we
    have
      $$ \ell_C (\Gamma, x) \ = \ \begin{cases}
         1, & \text{if \ $d=0$} \\
         0, & \text{if \ $d \ge 1$}. \end{cases} $$
    This follows from (\ref{eq:defshort}) since for $0 \le k \le d$, there are
    $2^{d-k} {d \choose k}$ faces $F \in \fF(C)$ of dimension $k$ and we have $h^{(sc)}
    (\Gamma_F, x) = 2^k$ for every such face.

    (b) Suppose that $\dim (C) = 1$. We have $h^{(sc)} (\Gamma, x) = tx+t+2$ and
    $\ell_C (\Gamma, x) = t(x+1)$, where $t$ is the number of interior vertices of
    $\Gamma$.

    (c) Part (a) of Figure \ref{fig:cubsub} shows a nongeometric cubical subdivision
    $\Gamma$ of a square $C$ ($\Gamma$ itself is shown as a geometric cubical complex).
    We have $h^{(sc)} (\Gamma, x) = 2x+6$ and $\ell_C (\Gamma, x) = 0$.

    (d) Let $C$ be a facet ($d$-dimensional face) of a $(d+1)$-dimensional cube $D$.
    The complex of faces of $D$ other than $C$ and $D$ defines a cubical subdivision
    $\Gamma$ of $C$ (this is the subdivision defined by the Schlegel diagram of $D$
    with respect to $C$; see \cite[Section 5.2]{Zie95}). The special case
    $d=2$ is shown in part (b) of Figure \ref{fig:cubsub}. Since for the boundary
    complex (complex of proper faces) $\partial D$ of $D$ we have $h^{(sc)} (\partial
    D, x) = 2^{d+1} (1 + x + \cdots + x^d)$
    (see, for instance, \cite[Lemma 1 (iv)]{Ad96}) and $\Gamma_F$ is the trivial
    subdivision of $F$ for every $F \in \fF(C) \sm \{C\}$, we may conclude that
      $$  h^{(sc)} (\Gamma_F, x) \ = \ \begin{cases}
         2^d \, (2 + 2x + \cdots + 2x^{d-1} + x^d), & \text{if \ $F=C$} \\
         2^{\dim(F)}, & \text{otherwise} \end{cases} $$
    for $F \in \fF(C)$. A simple calculation shows that $\ell_C (\Gamma, x) = 2^d (1+x)
    (1 + x + \cdots + x^{d-1})$. In the special case $d=2$ we have $h^{(sc)} (\Gamma,
    x) = 4x^2 + 8x + 8$ and $\ell_C (\Gamma, x) = 4(x+1)^2$.

    (e) Let $\Gamma$ be the complex of faces of two 3-dimensional cubes $C$ and $C'$,
    intersecting on a common facet $G$. Let $\sigma: \fF(\Gamma) \to \fF(C)$ be the
    subdivision map which pushes $C'$ into $C$, so that the facet of $C'$ opposite to
    $G$ ends up in the relative interior of $G$. Formally, we have $\sigma (G) = \sigma
    (C') = C$, $\sigma (F) = G$ for every proper face of $C'$ not contained in $G$
    and $\sigma (F) = F$ for all other nonempty faces of $\Gamma$. Thus, the
    restriction of this subdivision to the facet $G$ of $C$ is the one shown in part
    (b) of Figure \ref{fig:cubsub}, while the restriction to any other proper face of
    $C$ is the trivial subdivision. We have $h^{(sc)} (\Gamma, x) = 4(x+3)$ and $\ell_C
    (\Gamma, x) = -4x(x+1)$. The subdivision of part (a) of Figure \ref{fig:cubsub}
    is a two-dimensional version of this construction. \qed
  \end{example}

  \begin{remark} \label{rem:x+1}
    Let $\Gamma$ be a cubical subdivision of a cube $C$ of positive dimension $d$.
    We recall \cite[Lemma 1 (ii)]{Ad96} that $h^{(sc)} (\KR, -1) = 2^k \chi(\KR)$
    for every cubical complex $\KR$ of dimension $k$, where $\chi(\KR) = 1 +
    \widetilde{\chi} (\KR)$ is the Euler characteristic of $\KR$. Thus $h^{(sc)}
    (\Gamma_F, -1) = 2^k$ for every face $F \in \fF(C)$ of dimension $k$. It follows
    from (\ref{eq:defshort}) that $\ell_C (\Gamma, -1) = 0$ and hence the polynomial
    $\ell_C (\Gamma, x)$ is divisible by $x+1$.  \qed
  \end{remark}

  Let $\Gamma$ be a cubical subdivision of a cube $C$ with associated subdivision map
  $\sigma$. The \emph{excess} of a face $F \in \Gamma \sm \{\varnothing\}$ is defined
  as $e(F) := \dim \sigma(F) - \dim(F)$. The following statement is the analogue of
  \cite[Proposition 2.2]{Sta92} in our setting.

  \begin{proposition} \label{prop:excess}
    For every cubical subdivision $\Gamma$ of a $d$-dimensional cube $C$ we have
      \begin{equation} \label{eq:propexcess}
        \ell_C (\Gamma, x) \ = \ (-1)^d \sum_{F \in \Gamma \sm
        \{\varnothing\}} (-2)^{\dim(F)} \, x^{d-e(F)} (x-1)^{e(F)},
      \end{equation}
    where $e(F)$ is the excess of a nonempty face $F$ of $\Gamma$.
  \end{proposition}
  \begin{proof}
    We compute that
      \begin{eqnarray*}
      \ell_C (\Gamma, x) &=& \sum_{G \in \fF(C)} \ (-1)^{d - \dim(G)} \, h^{(sc)}
      (\Gamma_G, x) \\
      & & \\
      &=& \sum_{G \in \fF(C)} \ (-1)^{d - \dim(G)} \, \left( \, \sum_{F \in \Gamma_G
      \sm \{\varnothing\}} \ (2x)^{\dim(F)} \, (1-x)^{\dim(G) - \dim(F)} \right) \\
      & & \\
      &=& \sum_{F \in \Gamma \sm \{\varnothing\}} \ (-1)^d \left( \frac{2x}{1-x}
      \right)^{\dim(F)} \left( \, \sum_{G \in \fF(C): \, \sigma(F) \subseteq G} \
      (x-1)^{\dim(G)} \right)
      \end{eqnarray*}
    and observe that the inner sum in the last expression is equal to
    $(x-1)^{\dim \sigma(F)} x^{d - \dim \sigma(F)}$, since the interval $[\sigma(F),
    C]$ in $\fF(C)$ is a Boolean lattice of rank $d - \dim \sigma(F)$. The proposed
    formula follows.
  \end{proof}

  Next we consider products of cubical subdivisions. Let $\Gamma$ be a cubical
  subdivision of a $d$-dimensional cube $C$ with associated subdivision map $\sigma$
  and $\Gamma'$ be a cubical subdivision of a $d'$-dimensional cube $C'$ with
  associated subdivision map $\sigma'$. Then $\Gamma \times \Gamma'$ is a cubical
  subdivision of the $(d+d')$-dimensional cube $C \times C'$, if one defines the
  carrier of a face $G \times G'$ of $\Gamma \times \Gamma'$ as the product $\sigma(G)
  \times \sigma'(G')$.

  \begin{proposition} \label{prop:product}
    For cubical subdivisions $\Gamma$ and $\Gamma'$ of cubes $C$ and $C'$,
    respectively, we have $\ell_{C \times C'} (\Gamma \times \Gamma', x) = \ell_C
    (\Gamma, x) \ell_{C'} (\Gamma', x)$.
  \end{proposition}
  \begin{proof}
  By construction, the restriction of $\Gamma \times \Gamma'$ to a face $F \times F'$
  is equal to $\Gamma_F \times \Gamma'_{F'}$. Thus, recalling that $\fF(C \times C')
  = \fF(C) \times \fF(C')$ and using the defining equation (\ref{eq:defshort}) and
  (\ref{eq:prodsc}), we find that
    \begin{eqnarray*}
      \ell_{C \times C'} (\Gamma \times \Gamma', x)  &=& \sum_{F \times F' \in \fF(C
      \times C')} \ (-1)^{\dim(C \times C') - \dim(F \times F')} \,
        h^{(sc)} ( (\Gamma \times \Gamma')_{F \times F'}, x)  \\
      & & \\
      &=&  \sum_{F \times F' \in \fF(C)
      \times \fF(C')} \ (-1)^{d+d' - \dim(F) - \dim(F')} \,
        h^{(sc)} (\Gamma_F \times \Gamma'_{F'}, x)  \\
      & & \\
      &=& \sum_{F \in \fF(C), \, F' \in \fF(C')} \ (-1)^{d - \dim(F)} \,
        h^{(sc)} (\Gamma_F, x) \ (-1)^{d' - \dim(F')} \, h^{(sc)} (\Gamma'_{F'}, x) \\
      & & \\
      &=& \ell_C (\Gamma, x) \, \ell_{C'} (\Gamma', x),
      \end{eqnarray*}
    as promised.
  \end{proof}

  We end this section with some more examples and remarks.

  \begin{example} \label{ex:shortell}
    (a) Let $\ell_C (\Gamma, x) = \ell_0 + \ell_1 x + \cdots + \ell_d x^d$, where
     $\Gamma$ and $C$ are as in Definition \ref{def:short}. Since for $F \in \Gamma
     \sm \{\varnothing\}$ we have $e(F) = d$ if and only if $F$ is a vertex in the
     interior of $C$, it follows from Proposition \ref{prop:excess} that $\ell_0$
     is equal to the number of interior vertices of $\Gamma$. Alternatively, this
     follows from the equality
       $$ \ell_0 \ = \sum_{F \in \fF(C)} \ (-1)^{d - \dim(F)} \, h^{(sc)} (\Gamma_F,
       0) $$
     by a M\"obius inversion argument (see \cite[Proposition 3.7.1]{StaEC1}), since
     $h^{(sc)} (\KR, 0)$ is equal to the number of vertices of $\KR$ for every cubical
     complex $\KR$. The coefficient $\ell_d$ is equal to the coefficient of $x^d$ in
     $h^{(sc)} (\Gamma, x)$. By setting $x=0$ in
       $$ x^d \, h^{(sc)} (\Gamma, 1/x) \ = \ h^{(sc)} (\inte(\Gamma), x) $$
     (see Corollary \ref{cor:screc}), we deduce that $\ell_d$ is also equal to the
     number of interior vertices of $\Gamma$. Alternatively, this follows from our
     previous remark on $\ell_0$ and the symmetry of $\ell_C (\Gamma, x)$ (Theorem
     \ref{thm:symm}).

    (b) Similarly, from (\ref{eq:propexcess}) we deduce the formula $\ell_1 =
     -d f_0^\circ + 2 f_1^\circ - \tilde{f}_0$ for the coefficient of $x$ in $\ell_C
     (\Gamma, x)$, where (as in part (g) of \cite[Example 2.3]{Sta92}) $f_0^\circ$
     and $f_1^\circ$ are the numbers of interior vertices and edges of $\Gamma$,
     respectively, and $\tilde{f}_0$ is the number of vertices which lie in the
     relative interior of a facet of $C$.

    (c) Combining part (a) with Remark \ref{rem:x+1} we find that $\ell_C (\Gamma,
     x) = t(x+1)^2$ for every cubical subdivision $\Gamma$ of a square $C$ (the
     case $d=2$), where $t$ is the number of interior vertices of $\Gamma$.

    (d) Part (c) and Example \ref{ex:short} (b) imply that for $d \le 2$, the
     polynomial $\ell_C (\Gamma, x)$ depends only on the face poset (in fact, on the
     short cubical $h$-polynomial) of the complex $\Gamma$ (and not on $\sigma$).
     This is no longer true for $d \ge 3$. For instance, the complex $\Gamma$ of
     Example \ref{ex:short} (e) gives rise to a subdivision of the 3-dimensional
     cube $C$ which is the product of the trivial subdivision of a square and the
     subdivision of a line segment with one interior point. For this subdivision
     we have $\ell_C (\Gamma, x) = 0$. \qed
  \end{example}

  \section{Main properties of short cubical local $h$-vectors}
  \label{sec:main}

  This section develops the main properties of short cubical local $h$-vectors.
  From these properties, locality (Theorem \ref{thm:local}), symmetry (Theorem
  \ref{thm:symm}) and monotonicity (Corollary \ref{cor:monotone}) are proved in
  this section, while the proof of nonnegativity (Theorem \ref{thm:nonn}) is
  defered to Section \ref{sec:nonn}.

  The proof of our first result is a variation of that of Theorem 3.2 in
  \cite[Section 3]{Sta92}.

  \begin{theorem} \label{thm:local}
    Let $\KR$ be a pure cubical complex. For every cubical subdivision $\KR'$
    of $\KR$ we have
      \begin{equation} \label{eq:thmlocal}
        h^{(sc)} (\KR', x) \ = \sum_{F \in \fF(\KR)} \
        \ell_F (\KR'_F, x) \, h (\link_\KR (F), x).
      \end{equation}
  \end{theorem}
  \begin{proof}
  Denoting by $R(\KR', x)$ the right-hand side of (\ref{eq:thmlocal}) and setting
  $P = \fF(\KR)$, we compute that
    \begin{eqnarray*}
      R(\KR', x) &=& \sum_{G \in P} \ \ell_G (\KR'_G, x) \, h (\link_\KR (G), x) \\
      & & \\
      &=& \sum_{G \in P} \left( \, \sum_{F \le_P \, G} \ (-1)^{\dim(G)
      - \dim(F)} \, h^{(sc)} (\KR'_F, x) \right) h (\link_\KR (G), x) \\
      & & \\
      &=& \sum_{F \in P} \ h^{(sc)} (\KR'_F, x) \left( \, \sum_{F \le_P \, G}
      \ (-1)^{\dim(G) - \dim(F)} \, h (\link_\KR (G), x) \right) \\
      & & \\
      &=& - \, \sum_{F \in P} \left( \, \sum_{E \in \KR'_F \sm
      \{\varnothing\}} \ (2x)^{\dim(E)} \, (1-x)^{\dim(\KR) - \dim(E)} \right)
      \widetilde{\chi} (\link_\KR (F)) \\
      & & \\
      &=& - \sum_{E \in \KR' \sm \{\varnothing\}} \ (2x)^{\dim(E)} \,
      (1-x)^{\dim(\KR) - \dim(E)} \sum_{\sigma(E) \, \le_P \, F} \
      \widetilde{\chi} (\link_\KR (F)).
    \end{eqnarray*}

  \smallskip
  \noindent
  The fourth of the previous equalities follows from the defining equation
    $$ h^{(sc)} (\KR'_F, x) \ = \ \sum_{E \in \KR'_F \sm \{\varnothing\}} \
    (2x)^{\dim(E)} \, (1-x)^{\dim(F) - \dim(E)} $$
  and the equality
    $$ \sum_{F \le_P \, G} \ (-1)^{\dim(G) - \dim(F)} \, h (\link_\KR (G), x)
    \ = \ - (1-x)^{\dim(\KR) - \dim (F)} \, \widetilde{\chi} (\link_\KR (F)), $$
  for given $F \in P$. The latter is equivalent to
    $$ \sum_{F \le_P \, G} \ (-1)^{\dim(\KR) - \dim(G)} \, h (\link_\KR (G), x)
    \ = \ - (x-1)^{\dim(\KR) - \dim (F)} \, \widetilde{\chi} (\link_\KR (F)), $$
  which in turn follows by applying \cite[Lemma 3.1]{Sta92} to $\link_\KR (F)$,
  which is a pure simplicial complex of dimension $\dim(\KR) - \dim (F) - 1$.
  As in the proof of Theorem 3.2 in \cite[p.~813]{Sta92} we find that
    $$\sum_{\sigma(E) \, \le_P \, F} \ \widetilde{\chi} (\link_\KR (F)) \ =
    \ -1$$
  and hence
    $$ R(\KR', x) \ = \ \sum_{E \in \KR' \sm \{\varnothing\}} \ (2x)^{\dim(E)} \,
    (1-x)^{\dim(\KR) - \dim(E)} \ = \ h^{(sc)} (\KR', x), $$
  as claimed in the statement of the theorem.
  \end{proof}

  Our second result on short cubical local $h$-vectors is as follows.

  \begin{theorem} \label{thm:symm}
    For every cubical subdivision $\Gamma$ of a $d$-dimensional cube $C$ we have
      \begin{equation} \label{eq:thmsymm}
        x^d \, \ell_C (\Gamma, 1/x) \ = \ \ell_C (\Gamma, x).
      \end{equation}
    Equivalently, we have $\ell_i = \ell_{d-i}$ for $0 \le i \le d$, where
    $\ell_C (\Gamma) = (\ell_0, \ell_1,\dots,\ell_d)$.
  \end{theorem}
  \begin{proof}
    Setting $P = \fF(C)$ and using (\ref{eq:defshort}) and Corollary \ref{cor:screc}
    we find that
      \begin{eqnarray*}
      x^d \, \ell_C (\Gamma, 1/x) &=& \sum_{G \in P} \ (-x)^{d - \dim(G)} \,
      x^{\dim(G)} \, h^{(sc)} (\Gamma_G, 1/x) \\
      & & \\
      &=& \sum_{G \in P} \ (-x)^{d - \dim(G)} \, h^{(sc)} (\inte(\Gamma_G), x).
      \end{eqnarray*}
    We claim that
      \begin{equation} \label{eq:mobinvG}
        h^{(sc)} (\inte(\Gamma_G), x) \ = \ \sum_{F \, \le_P \, G} \ (x-1)^{\dim(G) -
        \dim(F)} \, h^{(sc)} (\Gamma_F, x)
      \end{equation}
    for every $G \in P$. Given the claim, we may conclude that
      \begin{eqnarray*}
      x^d \, \ell_C (\Gamma, 1/x) &=& \sum_{G \in P} \sum_{F \le_P G} \ 
      (-x)^{d - \dim(G)} \, (x-1)^{\dim(G) - \dim(F)} \, h^{(sc)} (\Gamma_F, x) \\
      & & \\
      &=& \sum_{F \in P} \ (-x)^{d - \dim(F)} \, h^{(sc)} (\Gamma_F, x)
      \sum_{F \le_P G \le_P C} \left( \frac{1-x}{x} \right)^{\dim(G) -
        \dim(F)} \\
      & & \\
      &=& \sum_{F \in P} \ (-x)^{d - \dim(F)} \, h^{(sc)} (\Gamma_F, x) \,
      \left( 1/x \right)^{d - \dim(F)} \\
      & & \\
      &=& \ell_C (\Gamma, x)
      \end{eqnarray*}
    where, in the third equality, we have used the fact that the interval $[F, C]$
    in $P$ is a Boolean lattice of rank $d - \dim(F)$. To verify (\ref{eq:mobinvG})
    we note that, in view of (\ref{eq:defsc}) and (\ref{eq:schinte}), we may write
      \begin{equation} \label{eq:GGa}
        (1-x)^{-\dim(G)} \ h^{(sc)} (\Gamma_G, x) \ = \ \sum_{E \in \Gamma_G \sm
        \{\varnothing\}} \ \left( \frac{2x}{1-x} \right)^{\dim(E)}
      \end{equation}
    and
      \begin{equation} \label{eq:GGb}
        (1-x)^{-\dim(G)} \ h^{(sc)} (\inte(\Gamma_G), x) \ = \ \sum_{E \in
        \inte(\Gamma_G)} \ \left( \frac{2x}{1-x} \right)^{\dim(E)}
      \end{equation}
    for $G \in P$. Denoting by $\beta(G)$ and $\alpha(G)$ the right-hand sides
    of (\ref{eq:GGa}) and (\ref{eq:GGb}), respectively, we have
      $$ \beta(G) \ = \, \sum_{F \, \le_P \, G} \, \alpha(F) $$
    for every $G \in P$. By M\"obius inversion \cite[Proposition 3.7.1]{StaEC1} on
    $P$ we get
      $$ \alpha(G) \ = \, \sum_{F \, \le_P \, G} \, (-1)^{\dim(G) - \dim(F)} \,
         \beta(F) $$
    for every $G \in P$. Replacing $\alpha(G)$ and $\beta(F)$ with the appropriate
    expressions from the left-hand sides of (\ref{eq:GGb}) and (\ref{eq:GGa})
    shows that the last equality is equivalent to (\ref{eq:mobinvG}).
  \end{proof}

  We recall \cite[Section 4]{Sta92} that a simplicial subdivision $\Delta'$ of a
  simplicial complex $\Delta$ is called \emph{quasi-geometric} if no face of
  $\Delta'$ has all its vertices lying in a face of $\Delta$ of smaller dimension.
  A cubical subdivision $\Gamma$ of a cube $C$ will be called \emph{locally
  quasi-geometric} if there are no faces $F \in \fF(C)$ and $G \in \fF(\Gamma)$
  and vertex $v$ of $G$, satisfying the following: (a) $F$ has smaller dimension 
  than $G$; and (b) for every edge $e$ of $G$ which contains $v$, the carrier of 
  $e$ is contained in $F$.
  Clearly, every geometric cubical subdivision of (a geometric cube) $C$ is
  locally quasi-geometric. Cubical subdivisions which are not locally quasi-geometric
  are given in parts (c) and (e) of Example \ref{ex:short}. A cubical subdivision
  $\KR'$ of a cubical complex $\KR$ will be called \emph{locally quasi-geometric}
  if for every $F \in \KR \sm \{\varnothing\}$, the restriction $\KR'_F$ is a
  locally quasi-geometric subdivision of $F$. Part (e) of Example \ref{ex:short}
  shows that the following statement fails to hold for all cubical subdivisions of
  cubes.
  \begin{theorem} \label{thm:nonn}
    For every locally quasi-geometric cubical subdivision $\Gamma$ of a
    $d$-dimensional cube $C$ we have $\ell_i \ge 0$ for $0 \le i \le d$, where
    $\ell_C (\Gamma) = (\ell_0, \ell_1,\dots,\ell_d)$.
  \end{theorem}

  The proof of Theorem \ref{thm:nonn} is given in Section \ref{sec:nonn}.

  \begin{question} \label{que:unimodal}
    Is $\ell_C (\Gamma)$ unimodal for every locally quasi-geometric cubical
    subdivision $\Gamma$ of a cube $C$?
  \end{question}

  \begin{remark} \label{rem:quasi}
    In analogy with the simplicial case, we may call a cubical subdivision $\Gamma$
    of $C$ \emph{quasi-geometric} if no face of $\Gamma$ has all its vertices
    lying in a face of $C$ of smaller dimension. Consider a square $C$ with vertices 
    $a, b, c, d$ and edges $\{a, b\}$, $\{b, c\}$, $\{c, d\}$, $\{a, d\}$. Insert two 
    points $e$ and $f$ in the relative interior of the edge $\{b, c\}$ and two points 
    $g$ and $h$ in the relative interior of the edge $\{c, d\}$, so that $e$ lies in
    the relative interior of $\{b, f\}$ and $g$ lies in the relative interior of $\{c, 
    h\}$. The subdivision of $C$ with three two-dimensional faces, having vertex sets 
    $\{a, b, d, e\}$, $\{d, e, g, h\}$ and $\{c, e, f, g\}$, is quasi-geometric but 
    not locally quasi-geometric. We leave it to the reader to construct a
    two-dimensional example of a locally quasi-geometric cubical subdivision which
    is not quasi-geometric. \qed
  \end{remark}

  The following statement is a short cubical analogue of \cite[Theorem 4.10]{Sta92}. All
  inequalities of the form $p(x) \ge q(x)$ which appear in the sequel, where $p(x)$ and
  $q(x)$ are real polynomials, will be meant to hold coefficientwise.
    \begin{corollary} \label{cor:monotone}
      Let $\KR$ be a cubical complex such that $\link_{\KR} (F)$ is Cohen-Macaulay,
      over some field, for every face $F \in \KR$ with $\dim(F) \ge 1$. Then we have
      $h^{(sc)} (\KR', x) \ge h^{(sc)} (\KR, x)$ for every locally quasi-geometric
      cubical subdivision $\KR'$ of $\KR$.
    \end{corollary}
    \begin{proof}
      In view of equation (\ref{eq:lemsc}), we may rewrite (\ref{eq:thmlocal}) as
        \begin{equation} \label{eq:thmlocaltwo}
         h^{(sc)} (\KR', x) \ = \ h^{(sc)} (\KR, x) \ \, + \sum_{F \in \KR: \, \dim(F)
        \ge 1} \ \ell_F (\KR'_F, x) \, h (\link_\KR (F), x).
       \end{equation}
    The result now follows from Theorem \ref{thm:nonn} and our assumptions on $\KR$.
    \end{proof}

    \begin{question} \label{que:monotone}
      Let $\KR$ be a Buchsbaum cubical complex, over some field. Does the conclusion
      of Corollary \ref{cor:monotone} hold for every cubical subdivision $\KR'$ of
      $\KR$?
    \end{question}

  We end this section with two specific kinds of cubical subdivision, to which Theorem
  \ref{thm:local} can be applied. The second one (see Example \ref{ex:cbs}) provided
  some of the motivation for this paper.

  \begin{example} \label{ex:shortstellar}
    Let $\KR$ be a cubical complex of positive dimension $d$ and let $\KR'$ be a
    cubical subdivision of $\KR$ for which the restriction $\KR'_F$ is the one in part
    (d) of Example \ref{ex:short}, for a given $d$-dimensional face $F$ of $\KR$, and
    the trivial subdivision for every other nonempty face of $\KR$. This subdivision,
    considered in \cite{Joc93} in the context of the lower bound problem for cubical
    polytopes, may be thought of as a cubical analogue of the stellar subdivision on
    a maximal face of a simplicial complex. The computations in Example \ref{ex:short}
    and (\ref{eq:thmlocaltwo}) imply that $h^{(sc)} (\KR', x) = h^{(sc)} (\KR, x) +
    2^d (1+x) (1 + x + \cdots + x^{d-1})$. \qed
  \end{example}

  \begin{example} \label{ex:cbs}
    Let $\KR$ be a $d$-dimensional cubical complex and $t$ be a positive integer. 
    Let $\KR'$ be the cubical subdivision of $\KR$ which has the following property: 
    the restriction $\KR'_F$ to any nonempty face $F$ of $\KR$ is combinatorially 
    isomorphic to the product of as many as $\dim(F)$ copies of the subdivision of a 
    line segment with $t$ interior vertices. For $t=1$, the face poset $\fF(\KR')$ 
    is isomorphic to the set of nonempty closed intervals in $\fF(\KR)$, ordered by 
    inclusion, and $\KR'$ is the \emph{cubical barycentric subdivision} (called the 
    barycentric cover in \cite{BBC97}) of $\KR$, studied in \cite{Sav10} (see also
    Example \ref{ex:cbs2}). We refer the reader to \cite[Section 2.3]{BBC97} and 
    \cite{Sav10} for more information on this special case.

    We can express the short cubical $h$-vector of $\KR'$ in terms of that of $\KR$ 
    as follows. By Example \ref{ex:short} (b) and Proposition \ref{prop:product}, we 
    have $\ell_F (\KR'_F, x) = t^{\dim(F)} (x+1)^{\dim(F)}$ for every $F \in \fF(\KR)$. 
    Thus, setting $P = \fF(\KR)$ and using Theorem \ref{thm:local}, we compute that
      \begin{eqnarray*}
      h^{(sc)} (\KR', x) &=& \sum_{F \in P} \
      (tx+t)^{\dim(F)} \, h (\link_\KR (F), x) \\
      & & \\
      &=& \sum_{F \in P} \ (tx+t)^{\dim(F)} \left( \,
      \sum_{F \le_P G} \ x^{\dim(G) - \dim(F)} (1-x)^{d - \dim(G)} \right) \\
      & & \\
      &=& (1-x)^d  \ \sum_{G \in P} \, \left( \frac{x}{1-x} \right)^{\dim(G)}
           \sum_{F \le_P G} \left( \frac{tx+t}{x} \right)^{\dim(F)}
           \\
      & & \\
      &=& (1-x)^d  \ \sum_{G \in P} \, \left( \frac{x}{1-x} \right)^{\dim(G)}
           \left( \frac{(t+2)x+t}{x} \right)^{\dim(G)} \\
      & & \\
      &=& (1-x)^d  \ \sum_{G \in P} \, \left( \frac{(t+2)x+t}{1-x} \right)^{\dim(G)} \\
      & & \\
      &=& (1-x)^d \ f_\KR \left( \frac{(t+2)x+t}{1-x} \right),
      \end{eqnarray*}
  where $f_\KR (x) = \sum_{F \in \fF(\KR)} \, x^{\dim(F)}$ is the $f$-polynomial of
  $\KR$ and for the fourth equality, we have used the fact that there are $2^{k-i}
  {k \choose i}$ faces of $G$ of dimension $i$ for $0 \le i \le k$, where $k =
  \dim(G)$. Expressing $f_\KR (x)$ in terms of $h^{(sc)} (\KR, x)$ (see, for instance,
  \cite[Equation (3)]{Sav10}), we find that
    \begin{equation} \label{eq:cbs}
       h^{(sc)} (\KR', x) \ = \ \left( \frac{tx+t+2}{2} \right)^d \,
       h^{(sc)} \left( \KR, \frac{(t+2)x+t}{tx+t+2} \right).
    \end{equation}
  The previous equality agrees with \cite[Theorem 3.2]{Sav10} in the special case $t=1$ 
  and thus recovers one of the main results of \cite{Sav10}.
    \qed
  \end{example}

  \section{Nonnegativity of short cubical local $h$-vectors}
  \label{sec:nonn}

  This section proves the nonnegativity of short cubical local $h$-vectors for
  locally quasi-geometric subdivisions (Theorem \ref{thm:nonn}). Our method follows
  that employed by Stanley to prove \cite[Theorem 4.6]{Sta92}. Thus we will assume
  familiarity with \cite[Section 4]{Sta92} and omit some of the details of the proof,
  giving emphasis to those points in which the arguments in \cite{Sta92} need to be
  adapted or generalized.

  We first observe that nonnegativity of the short cubical local $h$-vector is a
  consequence of the following statement.

  \begin{theorem} \label{thm:nonnv}
    For every locally quasi-geometric cubical subdivision $\Gamma$ of a
    $d$-dimensional cube $C$ and every vertex $v$ of $\Gamma$ we have
      \begin{equation} \label{eq:thmnonnv}
        \sum_{F \in \fF(C), v \in F}
        (-1)^{d - \dim(F)} \, h (\link_{\Gamma_F} (v), x) \, \ge \, 0.
      \end{equation}
  \end{theorem}

  \medskip
  \noindent
  \begin{proof}[Proof of Theorem \ref{thm:nonn}] Using (\ref{eq:lemsc}), we
  compute that
    \begin{eqnarray*}
     \ell_C (\Gamma, x) &=& \sum_{F \in \fF(C)} \ (-1)^{d - \dim(F)} \,
     h^{(sc)} (\Gamma_F, x) \\
     & & \\
     &=& \sum_{F \in \fF(C)} \ (-1)^{d - \dim(F)}
     \left( \, \sum_{v \in \vert(\Gamma_F)} h (\link_{\Gamma_F} (v), x)
     \right) \\
     & & \\
     &=& \sum_{v \in \vert(\Gamma)} \ \sum_{F \in \fF(C), v \in F}
        (-1)^{d - \dim(F)} \, h (\link_{\Gamma_F} (v), x). \\
    \end{eqnarray*}
  Thus, the result follows from Theorem \ref{thm:nonnv}.
  \end{proof}

  For the remainder of this section we let $\sigma: \fF(\Gamma) \to P = \fF(C)$ be the
  subdivision map associated with $\Gamma$. We also write $\dim(\sigma (v)) = d-e$,
  where $d = \dim (C)$ and $0 \le e \le d$. To motivate our approach towards
  proving Theorem \ref{thm:nonnv}, we consider two special cases. Suppose first
  that $e=d$, so that $v$ is a vertex of $C$. Assume further that $\link_\Gamma (v)$
  is a topological subdivision of the $(d-1)$-dimensional simplex $\link_C (v)$ (this
  happens, for instance, if $\Gamma$ is a geometric subdivision). Then $\link_\Gamma
  (v)$ is a quasi-geometric subdivision of $\link_C (v)$, in the sense of \cite[Section
  4]{Sta92}, and the left-hand side of (\ref{eq:thmnonnv}) is equal to the simplicial
  local $h$-vector of this subdivision \cite[Definition 2.1]{Sta92}. Therefore, Theorem
  \ref{thm:nonnv} follows in this case from \cite[Corollary 4.7]{Sta92}. On the other
  extreme, if $e=0$, so that $v$ is an interior vertex of $\Gamma$, then the left-hand
  side of (\ref{eq:thmnonnv}) consists of the single term $h (\link_\Gamma (v), x)$.
  This polynomial is nonnegative since $\link_\Gamma (v)$ is Cohen-Macaulay (in fact,
  a homology sphere) over any field. The proof which follows generalizes that in
  \cite[Section 4]{Sta92} and interpolates between these two special cases.

  \medskip
  \noindent
  \begin{proof}[Proof of Theorem \ref{thm:nonnv}]
  We fix a vertex $v$ of $\Gamma$ and set $\Delta = \link_\Gamma (v)$. We observe that
  the poset of faces $F \in \fF(C)$ with $v \in F$ is equal to the closed interval
  $[\sigma(v), C]$ in $\fF(C)$. This interval is a Boolean lattice of rank $e$.
  We denote (in the case $e \ge 1$) by $\Delta_1,
  \Delta_2,\dots,\Delta_e$ the subcomplexes of $\Delta$ of the form $\link_{\Gamma_G}
  (v)$, where $G$ runs through the codimension one faces of $C$ in the interval
  $[\sigma(v), C]$. We also write $\Delta_S = \cap_{i \in S} \, \Delta_i$ for $S
  \subseteq [e] := \{1, 2,\dots,e\}$, where $\Delta_\varnothing = \Delta$. Then the
  subcomplexes $\link_{\Gamma_F} (v)$ of $\Delta$ which appear in the left-hand side
  of (\ref{eq:thmnonnv}) are precisely the complexes $\Delta_S$ for $S \subseteq [e]$.
  Thus (\ref{eq:thmnonnv}) can be rewritten as
    \begin{equation} \label{eq:nonnv}
        \sum_{S \subseteq [e]} \ (-1)^{|S|} \, h (\Delta_S, x) \, \ge \, 0.
      \end{equation}
  Moreover, the complexes $\Gamma_F$ are Cohen-Macaulay over all fields (as they
  are homeomorphic to balls) for $F \in \fF(C)$ and therefore, so are the links of
  their faces and hence the complexes $\Delta_S$, for $S \subseteq [e]$.

  We fix an infinite field $\KK$ and consider the face ring (or Stanley-Reisner ring)
  $\KK[\Delta]$ of $\Delta$; see \cite[Chapter II]{StaCCA}. Given a linear system of
  parameters $\Theta = \{\Theta_1,\dots,\Theta_d\}$ for $\KK[\Delta]$, we set $\KK 
  (\Delta) := \KK[\Delta] / (\Theta)$ and denote by $L_v (\Gamma)$ the image in $\KK
  (\Delta)$ of the ideal of $\KK[\Delta]$ generated by the face monomials $x^F$, where 
  $F \in \Delta$ runs through all faces which do not belong to any of the $\Delta_i$. 
  It follows from \cite[Theorem 4.6]{Sta92} that, in the special case $e=d$, the system
  $\Theta$ can be chosen so that
    \begin{equation} \label{eq:hilbert}
      \sum_{S \subseteq [e]} \, (-1)^{|S|} \, h (\Delta_S, x) \ = \ \sum_{i=0}^d \,
      \dim_\KK L_v (\Gamma)_i \, x^i,
    \end{equation}
  where $L_v (\Gamma)_i$ denotes the homogeneous part of $L_v (\Gamma)$ of degree
  $i$. We note that (\ref{eq:hilbert}) also holds when $e=0$ (for every $\Theta$),
  since then $L_v (\Gamma) = \KK (\Delta)$, the left-hand side of (\ref{eq:hilbert}) 
  is equal to $h (\Delta, x)$ and Cohen-Macaulayness of $\Delta$ over $\KK$ implies
  that $\dim_\KK \KK(\Delta)_i = h_i (\Delta)$ for all $0 \le i \le d$.

  To establish (\ref{eq:nonnv}) in the general case, it suffices to show that 
  $\Theta$ can always be chosen so that (\ref{eq:hilbert}) holds. We call a sequence
  $\{\Theta_1,\dots,\Theta_d\}$ of linear forms on $\KK[\Delta]$ \emph{special} if
  $\Theta_{d-e+i}$ is a linear combination of vertices of $\Delta$ which do not
  belong to $\Delta_i$, for each $1 \le i \le e$. Since $\Gamma$ is locally
  quasi-geometric, there is no face $F$ of $\Delta$ and set $S \subseteq [e]$ of
  cardinality exceeding $d - \dim(F)$, such that all vertices of $F$ belong to
  $\Delta_S$. Thus, by choosing $\{\Theta_1,\dots,\Theta_d\}$ as generically as
  possible, subject to the condition of being special, it follows as in the proof
  of \cite[Corollary 4.4]{Sta92} that a special linear system of parameters for
  $\KK[\Delta]$ exists. We will show that (\ref{eq:hilbert}) holds for such $\Theta$.
  We consider the natural complex of $\KK[\Delta]$-modules
    \begin{equation} \label{eq:complex}
      \KK[\Delta] \to \bigoplus_{i=1}^e \ \KK[\Delta] / N_i \to
      \bigoplus_{1 \le i < j \le e} \KK[\Delta] / (N_i + N_j) \to \cdots \to
      \KK[\Delta] / (N_1 + \cdots + N_e) \to 0,
    \end{equation}
  where $N_i$ is the ideal of $\KK[\Delta]$ generated by all face monomials $x^F$
  for which $F \in \Delta$ does not belong to $\Delta_i$. The complex
  (\ref{eq:complex}) naturally generalizes that in \cite[Lemma 4.9]{Sta92}. Clearly,
  for $1 \le i_1 < \cdots < i_r \le e$ we have
    \begin{equation} \label{eq:iso}
      \KK[\Delta] / (N_{i_1} + \cdots + N_{i_r}) \ \cong \ \KK[\Delta_S]
    \end{equation}
  as $\KK[\Delta]$-modules, where $S = [e] \sm \{i_1,\dots,i_r\}$. Therefore
  (\ref{eq:complex}) can be rewritten as
    \begin{equation} \label{eq:complexS}
      \KK[\Delta] \, \to \, \bigoplus_{i=1}^e \ \KK[\Delta_i] \, \to \cdots \to \,
      \bigoplus_{S \subseteq [e]: \, |S| = k} \, \KK[\Delta_S] \, \to \cdots \to \,
      \KK[\Delta_{[e]}] \, \to 0.
    \end{equation}
  We claim that taking the quotient of (\ref{eq:complexS}) by $(\Theta)$ produces an 
  exact complex, namely
    \begin{equation} \label{eq:complexT}
      \KK(\Delta) \, \to \, \bigoplus_{i=1}^e \ \KK(\Delta_i) \, \to \cdots \to \,
      \bigoplus_{S \subseteq [e]: \, |S| = k} \, \KK(\Delta_S) \, \to \cdots \to \,
      \KK(\Delta_{[e]}) \, \to 0,
    \end{equation}
  where $\KK(\Delta_S) := \KK[\Delta_S] / (\Theta)$ for $S \subseteq [e]$. Given the
  claim, the proof concludes as follows. Except for the irrelevant terms $\Theta_i$
  which annihilate $\KK[\Delta_S]$ (specifically, those with $i \in S$), the
  sequence $\Theta$ is a linear system of parameters for $\KK[\Delta_S]$. Since each
  complex $\Delta_S$ is Cohen-Macaulay over $\KK$, we have
    $$ \sum_{i=0}^d \, \dim_\KK \KK(\Delta_S)_i \, x^i \ = \ h (\Delta_S, x) $$
  for every $S \subseteq [e]$. Moreover, if $\delta$ is the leftmost map in
  (\ref{eq:complexT}), then $\ker (\delta) = L_v (\Gamma)$ and hence (\ref{eq:hilbert})
  follows by computing the Hilbert series of $\ker (\delta)$ from the exact sequence
  (\ref{eq:complexT}).

  To prove the claim we show by induction on $r$, as in the proof of \cite[Lemma
  4.9]{Sta92}, that the sequence
    \begin{equation} \label{eq:complexU}
      \overline{\KK[\Delta]} \, \to \, \bigoplus_{i=1}^e \ \overline{\KK[\Delta_i]} \,
      \to \cdots \to \, \bigoplus_{S \subseteq [e]: \, |S| = k} \, \overline{\KK[\Delta_S]}
      \, \to \cdots \to \, \overline{\KK[\Delta_{[e]}]} \, \to 0,
    \end{equation}
  obtained by taking the quotient of (\ref{eq:complexS}) by the ideal 
  $(\Theta_1,\dots,\Theta_r)$, is exact. For $r=0$, this means that the sequence
  (\ref{eq:complexS}) itself is exact and the argument given on \cite[p.~817]{Sta92}
  works (with the understanding that the dimension of the simplex $\Sigma$,
  mentioned there, is one less than the number of indices $1 \le i \le e$ with
  $F \in \Delta_i$). For the inductive step, we distinguish two cases. First, if $0
  \le r < d-e$, then $\Theta_{r+1}$ is a nonzero-divisor on each term of
  (\ref{eq:complexU}) and hence all kernels $B^k$ of the maps which are given by
  multiplication with $\Theta_{r+1}$ on these terms are identically zero. Otherwise,
  we write $r+1 = d-e+j$ for some $1 \le j \le e$ and note that $\Theta_{r+1}$ either
  annihilates $\overline{\KK[\Delta_S]}$ or is a nonzero-divisor on it, depending on
  whether $j \in S$ or not. Thus we have
    $$ B^k \ = \ \bigoplus_{\scriptsize \begin{array}{c} S \subseteq [e]: |S| = k,
       \\ j \in S \end{array}} \, \overline{\KK[\Delta_S]} $$
  for $0 \le k \le d$ (for instance, $B^0 = 0$ is the kernel of $\Theta_{r+1}$ on
  $\overline{\KK[\Delta]}$) and the argument given in the proof of \cite[Lemma
  4.9]{Sta92} to show that $B^1 \to \cdots \to B^d \to 0$ is exact goes through, as
  does the last step of the proof, completing the induction.
  \end{proof}

  \begin{remark} \label{rem:modules}
    Using (\ref{eq:lemsc}), one can show as in the proof of Theorem \ref{thm:symm}
    that each contribution (\ref{eq:thmnonnv}) to $\ell_C (\Gamma, x)$ is a polynomial
    with symmetric coefficients. Moreover, it follows as in \cite[Corollary
    4.19]{Sta92} that $L_v (\Gamma)$ is a Gorenstein $\KK[\Delta]$-module for every
    vertex $v$ of $\Gamma$. It might be worth studying further properties and
    examples of the modules $L_v (\Gamma)$. \qed
  \end{remark}

  \begin{remark} \label{rem:regular}
    It was shown by Stanley \cite[Theorem 5.2]{Sta92} that the simplicial local
    $h$-vector is unimodal for every \emph{regular} (hence geometric) triangulation
    of a simplex. It is plausible that the techniques of \cite[Section 5]{Sta92} can
    be applied to show that if $\Gamma$ is a regular cubical subdivision of the
    geometric cube $C$, then each contribution (\ref{eq:thmnonnv}) to $\ell_C
    (\Gamma, x)$ is unimodal. This would provide an affirmative answer to Question
    \ref{que:unimodal} in the special case of regular cubical subdivisions. \qed
  \end{remark}

  \section{The long cubical local $h$-vector}
  \label{sec:long}

  This section defines (long) cubical local $h$-vectors, studies their elementary
  properties and lists some open questions related to them.

  \begin{definition} \label{def:long}
    Let $C$ be a $d$-dimensional cube. For any cubical subdivision $\Gamma$ of
    $C$, we define a polynomial $\LR_C (\Gamma, x) = \LR_0 + \LR_1 x + \cdots +
    \LR_{d+1} x^{d+1}$ by
      \begin{equation} \label{eq:deflong}
        \LR_C (\Gamma, x) \ = \sum_{F \in \fF(C)} \ (-1)^{d - \dim(F)} \, h^{(c)}
        (\Gamma_F, x).
      \end{equation}
    We call $\LR_C (\Gamma, x)$ the (long) \emph{cubical local $h$-polynomial} of
    $\Gamma$ (with respect to $C$). We call $\LR_C (\Gamma) = (\LR_0,
    \LR_1,\dots,\LR_{d+1})$ the (long) \emph{cubical local $h$-vector} of $\Gamma$
    (with respect to $C$).
  \end{definition}

  \begin{remark} \label{rem:long}
    (i) It follows from (\ref{eq:hd+1}) that the coefficient $\LR_{d+1}$ of $x^{d+1}$
    in the right-hand side of (\ref{eq:deflong}) is equal to zero. Hence the degree of
    $\LR_C (\Gamma, x)$ cannot exceed $d$.

    (ii) Suppose that $d \ge 1$. It follows from (\ref{eq:h0}) that the constant
    term $\LR_0$ of the right-hand side of (\ref{eq:deflong}) is equal to zero.
    Hence $\LR_C (\Gamma, x)$ is divisible by $x$. \qed
  \end{remark}

  The two kinds of cubical local $h$-vectors we have defined are related in a simple
  way, as the following proposition shows.
    \begin{proposition} \label{prop:long}
      For every cubical subdivision $\Gamma$ of a cube $C$ of dimension $d \ge 1$
      we have
        \begin{equation} \label{eq:proplong}
          x \, \ell_C (\Gamma, x) \ = \ (x+1) \, \LR_C (\Gamma, x).
        \end{equation}
      Thus we have $\LR_0 = \LR_{d+1} = 0$ and $\LR_{i+1} = \ell_i - \ell_{i-1} +
      \cdots + (-1)^i \ell_0$ for $0 \le i \le d-1$, where $\LR_C (\Gamma) = (\LR_0,
      \LR_1,\dots,\LR_{d+1})$ and $\ell_C (\Gamma) = (\ell_0, \ell_1,\dots,\ell_d)$.
    \end{proposition}
    \begin{proof}
      Since $\widetilde{\chi} (\Gamma_F) = 0$ for $F \in \fF(C)$, applying
      (\ref{eq:defc}) to $\Gamma_F$ we get
        $$ (x+1) \, h^{(c)} (\Gamma_F, x) \ = \ 2^{\dim(F)} + x h^{(sc)} (\Gamma_F,
           x). $$
      The result follows by multiplying (\ref{eq:deflong}) by $x+1$ and using the
      previous equality.
    \end{proof}

  The statements in the next example follow from Proposition \ref{prop:long} and the
  computations in Examples \ref{ex:short} and \ref{ex:shortell}.

    \begin{example} \label{ex:long}
      (a) For the trivial subdivision of the $d$-dimensional cube $C$, the cubical
      local $h$-polynomial is equal to 1, if $d=0$, and vanishes otherwise. For a
      general cubical subdivision $\Gamma$ of $C$ we have $\LR_C (\Gamma, x) = tx$,
      if $d=1$, and $\LR_C (\Gamma, x) = tx(x+1)$, if $d=2$, where $t$ is the number
      of interior vertices of $\Gamma$. For the subdivision in part (d) of Example
      \ref{ex:short} we have $\LR_C (\Gamma, x) = 2^d (x + x^2 + \cdots + x^d)$.

      (b) Let $\LR_C (\Gamma) = (\LR_0, \LR_1,\dots,\LR_{d+1})$, where $\Gamma$ and
       $C$ are as in Definition \ref{def:long}. Then $\LR_1 = \LR_d$ is equal to the
       number of interior vertices of $\Gamma$ and $\LR_2 = -(d+1) f_0^\circ + 2
       f_1^\circ - \tilde{f}_0$, where the notation is as in part (b) of Example
       \ref{ex:shortell}.  \qed
    \end{example}

    \begin{corollary} \label{cor:symm}
      For every cubical subdivision $\Gamma$ of a cube $C$ of dimension $d \ge 1$
      we have
        \begin{equation} \label{eq:corsymm}
          x^{d+1} \, \LR_C (\Gamma, 1/x) \ = \ \LR_C (\Gamma, x).
        \end{equation}
      Equivalently, we have $\LR_i = \LR_{d+1-i}$ for $0 \le i \le d+1$, where
      $\LR_C (\Gamma) = (\LR_0, \LR_1,\dots,\LR_{d+1})$.
    \end{corollary}
    \begin{proof}
      This follows from Theorem \ref{thm:symm} and Proposition \ref{prop:long}.
    \end{proof}

  The following statement is the analogue of Theorem \ref{thm:local} for
  the long cubical $h$-vector.

  \begin{theorem} \label{thm:clocal}
    Let $\KR$ be a pure cubical complex. For every cubical subdivision $\KR'$
    of $\KR$ we have
      \begin{equation} \label{eq:thmclocal}
        h^{(c)} (\KR', x) \ = \ h^{(c)} (\KR, x) \ + \sum_{F \in \KR: \, \dim(F)
        \ge 1} \LR_F (\KR'_F, x) \, h (\link_\KR (F), x).
      \end{equation}
  \end{theorem}
  \begin{proof}
    We multiply (\ref{eq:thmlocal}) by $x$ and use Proposition \ref{prop:long},
    equation (\ref{eq:lemsc}) and the fact that $\ell_F (\KR'_F, x) = 1$ for
    each zero-dimensional face $F \in \KR$ to get
      \begin{eqnarray*}
      x \, h^{(sc)} (\KR', x) &=& (x+1) \sum_{F \in \KR: \, \dim(F) \ge 1}
      \LR_F (\KR'_F, x) \, h (\link_\KR (F), x) \ + \ x \sum_{v \in \vert(\KR)}
      h (\link_\KR (v), x) \\
      & & \\
      &=& (x+1) \sum_{F \in \KR: \, \dim(F) \ge 1} \LR_F (\KR'_F, x) \, h (\link_\KR
          (F), x) \ + \ x \, h^{(sc)} (\KR, x).
      \end{eqnarray*}
    Applying (\ref{eq:defc}) to $\KR$ and $\KR'$ and noting that $\widetilde{\chi}
    (\KR) = \widetilde{\chi} (\KR')$ we get
      $$  (x+1) \, h^{(c)} (\KR', x) \, - \, x \, h^{(sc)} (\KR', x) \ = \
          (x+1) \, h^{(c)} (\KR, x) \, - \, x \, h^{(sc)} (\KR, x). $$
    The result follows by adding the previous two equalities and dividing by
    $x+1$.
  \end{proof}

  \begin{example} \label{ex:longstellar}
    For the cubical stellar subdivision $\KR'$ of Example \ref{ex:shortstellar} we
    have $h^{(c)} (\KR', x) = h^{(c)} (\KR, x) + 2^d (x + x^2 + \cdots + x^d)$. \qed
  \end{example}

  The following are the main open questions on cubical local $h$-vectors. In view
  of Theorem \ref{thm:clocal}, a positive answer to the first question would imply
  a positive answer to the third one for locally quasi-geometric subdivisions. Since
  we have $\LR_C (\Gamma, x) = -4x^2$ for the subdivision in part (e) of Example
  \ref{ex:short}, the first question has a negative answer for general
  subdivisions.

  \begin{question} \label{que:cnonn}
    Does $\LR_C (\Gamma, x) \ge 0$ hold for every locally quasi-geometric cubical
    subdivision $\Gamma$ of a cube $C$?
  \end{question}

  \begin{question} \label{que:cunimodal}
    Is $\LR_C (\Gamma)$ unimodal for every locally quasi-geometric cubical
    subdivision $\Gamma$ of a cube $C$?
  \end{question}

  \begin{question} \label{que:cmonotone}
    Does $h^{(c)} (\KR', x) \ge h^{(c)} (\KR, x)$ hold for every cubical subdivision
    $\KR'$ of a Buchsbaum cubical complex $\KR$?
  \end{question}

  Example \ref{ex:long} and Theorem \ref{thm:clocal} make it clear that Question
  \ref{que:cmonotone} has an affirmative answer for any cubical subdivision $\KR'$ of
  any cubical complex $\KR$ of dimension at most two. A partial result related to
  Question \ref{que:cnonn} is the fact (which follows from Theorem \ref{thm:nonn}
  and Proposition \ref{prop:long}) that $(x+1) \, \LR_C (\Gamma, x) \ge 0$ holds for
  every locally quasi-geometric cubical subdivision $\Gamma$ of a cube $C$.

  \section{Generalizations}
  \label{sec:gen}

  This section explains how the theory of short cubical local $h$-vectors can be
  extended to more general types of subdivisions, in the spirit of \cite[Part
  II]{Sta92}, namely to subdivisions of locally Eulerian posets. Basic examples
  are provided by CW-regular subdivisions of a family of complexes which includes
  all regular CW-complexes. The motivating special case is that of a cubical regular
  CW-complex which subdivides a cube and does not necessarily have the intersection
  property. As a byproduct of this generalization, it is shown that the theory of
  \cite[Part I]{Sta92} applies to other types of $h$-vectors of simplicial complexes,
  such as the short simplicial $h$-vector of \cite{HN02} and certain natural
  generalizations.

  Throughout this section, we will assume familiarity with Sections 6 and 7 of
  \cite{Sta92}. The proofs of the results in this section are straightforward
  generalizations of proofs of corresponding results in \cite{Sta92} and thus, most
  of them will be omitted.

  \subsection{Formal subdivisions} \label{subsec:formal}
  We will call a poset $P$ \emph{lower graded} if the principal order ideal $P_{\le
  t} = \{s \in P: s \le_P t\}$ is finite and graded for every $t \in P$. Contrary to
  the convention of \cite{Sta92}, such a poset need not have a minimum element. We
  denote by $\rho (t)$ the rank (common length of all maximal chains) of $P_{\le t}$.
  Our primary example of a lower graded poset will be the poset of nonempty faces of
  a (cubical) regular CW-complex.

  As in \cite[Section 6]{Sta92}, we will work with the space $\KK[x]^P$ of
  functions $f: P \to \KK[x]$, where $\KK$ is a fixed field, and with the incidence
  algebra $I(P)$ of functions $g: \Inte(P) \to \KK[x]$, where $\Inte(P)$ stands
  for the set of all (nonempty) closed intervals of $P$. We will write $f_t (x)$ and
  $g_{st} (x)$ for the value of $f \in \KK[x]^P$ and $g \in I(P)$, respectively, on
  an element $t \in P$ and interval $[s, t] \in \Inte(P)$. The (convolution) product
  on $I(P)$ is defined whenever $P$ is locally finite (meaning that each element of
  $\Inte(P)$ is finite). For such $P$, a function $\kappa \in I(P)$ is called
  \emph{unitary} if $\kappa_{tt} = 1$ for every $t \in P$.

  Assuming that $P$ is lower graded, we write, as in \cite{Sta92},
    $$ \overline{f}_t (x) \ = \ x^{\rho(t)} f_t (1/x), \ \ \ \ t \in P $$
  for every $f \in \KK[x]^P$ for which $\deg f_t (x) \le \rho(t)$ holds for all $t \in
  P$. Similarly, assuming that $P$ is locally graded (meaning that each element of
  $\Inte(P)$ is finite and graded), we write $\overline{g}_{st} (x)
  = x^{\rho(s, t)} g_{st} (1/x)$ for every $g \in I(P)$ for which $\deg g_{st} (x) \le
  \rho(s, t)$ holds for all $[s, t] \in \Inte(P)$, where $\rho(s, t)$ denotes the rank
  of the interval $[s, t]$. The following definition is \cite[Definition 6.2]{Sta92},
  where the assumption that $P$ has a minimum element has been removed from part (a).

  \begin{definition} \label{def:accept}
    Let $P$ be a locally graded poset and $\kappa \in I(P)$ be unitary.
     \begin{itemize}
       \item[(a)] Assume that $P$ is lower graded. A function $f: P \to \KK[x]$ is
       called \emph{$\kappa$-acceptable} if $f \kappa = \overline{f}$, i.e., if
         \begin{equation} \label{eq:accept1}
           \sum_{s \le_P \, t} \ f_s (x) \kappa_{st} (x) \ = \ \overline{f}_t (x)
         \end{equation}
       for every $t \in P$.
       \item[(b)] A function $g \in I(P)$ is called \emph{$\kappa$-totally acceptable}
       if $g \kappa = \overline{g}$, i.e., if
        \begin{equation} \label{eq:accept2}
          \sum_{s \le_P \, t \le_P \, u} \ g_{st} (x) \kappa_{tu} (x) \ = \
          \overline{g}_{su} (x)
        \end{equation}
      for all $s \le_P t$.
      \end{itemize}
  \end{definition}

  Given a locally finite poset $P$, we recall (see \cite[Theorem 6.5]{Sta92}) that a
  unitary function $\kappa \in I(P)$ is a \emph{$P$-kernel} if and only if
  $\overline{\kappa} = \kappa^{-1}$. Part (a) of the following statement generalizes
  part (a) of \cite[Corollary 6.7]{Sta92}.

    \begin{proposition} \label{prop:uaccept}
      Let $P$ be a locally graded poset with $P$-kernel $\kappa$.
        \begin{itemize}
        \item[(a)] Assume that $P$ is lower graded. Then there exists a unique
        $\kappa$-acceptable function $\gamma = \gamma(P, \kappa): P \to \KK[x]$
        satisfying: {\rm (i)} $\gamma_a (x) = 1$ for every minimal element $a \in P$,
        and {\rm (ii)} $\deg \gamma_t (x) \le \lfloor (\rho(t) - 1) / 2 \rfloor$
        for every $t \in P$ with $\rho (t) \ge 1$.
        \item[(b)] {\rm (\cite[Corollary 6.7 (b)]{Sta92})} There exists a unique
        $\kappa$-totally acceptable function $\xi = \xi(P, \kappa) \in I(P)$ satisfying:
        {\rm (i)} $\xi_{tt} (x) = 1$ for every $t \in P$, and {\rm (ii)} $\deg \xi_{st}
        (x) \le \lfloor (\rho(s, t) - 1) / 2 \rfloor$ for all $s <_P t$.
        \item[(c)] Assume that $P$ is lower graded. Then the functions $\gamma =
        \gamma(P, \kappa)$ and $\xi = \xi(P, \kappa)$ are related by the equation
          \begin{equation} \label{eq:propuaccept}
            \gamma_t (x) \ = \ \sum_{a \in \min (P_{\le t})} \ \xi_{at} (x)
          \end{equation}
        for $t \in P$, where $\min(P_{\le t})$ stands for the set of minimal elements
        of $P_{\le t}$.
        \end{itemize}
   \end{proposition}
   \begin{proof}
     The proof of part (a) of \cite[Corollary 6.7]{Sta92} (as well as those of Lemma
     6.4 and Proposition 6.6 in \cite{Sta92}, on which this proof is based) does not
     use the assumption that $P$ has a minimum element. Thus part (a) of the proposition
     holds as well. Part (b) is identical to part (b) of \cite[Corollary 6.7]{Sta92}.
     For part (c), let $\delta_t (x)$ denote the right-hand side of (\ref{eq:propuaccept}).
     Clearly, we have $\delta_a (x) = 1$ for every minimal element $a \in P$ and $\deg
     \delta_t (x) \le \lfloor (\rho(t) - 1) / 2 \rfloor$ for every $t \in P$ with $\rho
     (t) \ge 1$. Moreover, for $t \in P$ we have
       \begin{eqnarray*}
       \sum_{s \le_P \, t} \ \delta_s (x) \kappa_{st} (x) &=&
       \sum_{s \le_P \, t} \, \sum_{a \in \min (P_{\le s})} \ \xi_{as} (x) \kappa_{st}
       (x) \ = \ \sum_{a \in \min (P_{\le t})} \, \sum_{a \le_P \, s \le_P \, t} \
       \xi_{as} (x) \kappa_{st} (x) \\
       & & \\
       &=& \sum_{a \in \min (P_{\le t})} \ \overline{\xi}_{at} (x) \ = \
       \overline{\delta}_t (x).
       \end{eqnarray*}
     Thus $\delta: P \to \KK[x]$ is $\kappa$-acceptable and the result follows from
     the uniqueness statement of part (a).
   \end{proof}

  For the remainder of this section we consider a lower graded, locally Eulerian poset
  $P$ (so that the M\"obius function of $P$ satisfies $\mu_P (s, t) = (-1)^{\rho(s,
  t)}$ for $s \le_P t$) and fix the $P$-kernel $\lambda \in I(P)$ with $\lambda (s, t)
  = (x-1)^{\rho(s, t)}$ for $s \le_P t$ (see \cite[Proposition 7.1]{Sta92}).

  \begin{example} \label{ex:gammacube}
  Suppose that every closed interval in $P$ is isomoprhic to a Boolean lattice. Then
  we have $\xi_{st} (x) = 1$ for all $s \le_P t$ \cite[Proposition 2.1]{Sta87}
  \cite[Example 3.14.8]{StaEC1} and hence $\gamma_t (x)$ is equal to the number of
  minimal elements of $P_{\le t}$ by (\ref{eq:propuaccept}). Thus, if $P_{\le u}$ is
  isomorphic to the poset of nonempty faces of a simplex for some $u \in P$, then
  $\gamma_t (x) = \rho(t) + 1$ for $t \le_P u$. Similarly, if $P_{\le u}$ is
  isomorphic to the poset of nonempty faces of a cube, then $\gamma_t (x) =
  2^{\rho(t)}$ for $t \le_P u$. \qed
  \end{example}

  We will refer to the maximum length of a chain in a finite poset $P$ as the
  \emph{length} of $P$. The following definition generalizes that of the $h$-vector of
  a finite lower Eulerian poset in \cite[Example 7.2]{Sta92}.

  \begin{definition} \label{def:hpoly}
    Let $P$ be a finite, lower graded and locally Eulerian poset of length $d$ and
    let $\gamma = \gamma(P, \lambda): P \to \KK[x]$. The polynomial $h (P, x)$
    defined by
      \begin{equation} \label{eq:def-hgen}
      \overline{h} (P, x) \ = \ x^d \, h (P, 1/x) \ = \ \sum_{t \in P} \ \gamma_t
      (x) \, (x-1)^{d - \rho(t)}
      \end{equation}
    is called the \emph{$h$-polynomial} of $P$. The \emph{$h$-vector} of $P$ is the
    sequence $(h_0, h_1,\dots,h_d)$, where $h (P, x) = h_0 + h_1 x + \cdots + h_d x^d$.
  \end{definition}

  Since $\gamma$ is $\lambda$-acceptable, we have $h (P, x) = \gamma_{\hat{1}} (x)$
  when $P$ has a maximum element $\hat{1}$. The following example lists several known
  notions of $h$-vectors of complexes and posets which are captured by Definition
  \ref{def:hpoly}.

  \begin{example} \label{ex:cchgen}
  (a) Suppose that $P$ is finite of length $d$ and has a minimum element. Then $\gamma$
  is as in \cite[Section 7]{Sta92} and $h (P, x)$ coincides with the generalized
  $h$-polynomial of $P$, introduced in \cite[Section 4]{Sta87} (see also \cite[Example
  7.2]{Sta92}). If, in addition, every principal order ideal of $P$ is isomorphic to a
  Boolean lattice (so that $P$ is a simplicial poset, in the sense of 
  \cite[p.~135]{StaEC1} \cite[p.~112]{StaCCA}), then $\gamma_t (x) = 1$ for every $t 
  \in P$ and hence (see \cite[Corollary 2.2]{Sta87})
    \begin{equation} \label{eq:hsp}
      h (P, x) \ = \ x^d \ \sum_{t \in P} \ \gamma_t (1/x) \left( \frac{1}{x} - 1
       \right)^{d - \rho(t)} = \ \sum_{t \in P} \ x^{\rho(t)} \,
       (1-x)^{d - \rho(t)}
    \end{equation}
  is the $h$-polynomial of the simplicial poset $P$. We refer the reader to \cite[Section 
  III.6]{StaCCA} for further information on $h$-polynomials (equivalently, on $h$-vectors) 
  of simplicial posets. 

  (b) Suppose that $P$ is finite of length $d-1$ and that every principal order ideal
  of $P$ is isomorphic to the poset of nonempty faces of a simplex. In view of Example
  \ref{ex:gammacube}, we have
    $$ h (P, x) \ = \ x^{d-1} \ \sum_{t \in P} \ \gamma_t (1/x) \left( \frac{1}{x} - 1
       \right)^{d - \rho(t) - 1} = \ \sum_{t \in P} \ (\rho(t) + 1) \, x^{\rho(t)} \,
       (1-x)^{d - \rho(t) - 1}. $$
  Thus, if $P$ is the poset of nonempty faces of a pure simplicial complex $\Delta$ of
  dimension $d-1$, then
     \begin{eqnarray*}
      h (P, x) &=& \sum_{F \in \Delta \sm \{\varnothing\}} \ |F| \cdot x^{|F|-1} \,
      (1-x)^{d - |F|} \ = \ \sum_{v \in \vert (\Delta)} \ \sum_{E \in \link_\Delta (v)}
      \ x^{|E|} \cdot (1-x)^{d - |E| - 1} \\
      & & \\
      &=& \sum_{v \in \vert (\Delta)} \ h(\link_\Delta (v), x)
    \end{eqnarray*}
  is the short simplicial $h$-polynomial of $\Delta$, introduced by Hersh and Novik
  \cite{HN02}. More generally, if $m \in \NN$ is fixed and $P$ is the poset of faces of
  $\Delta$ of dimension $m$ or higher, then $h (P, x)$ is equal to the sum of the
  $h$-polynomials of the links of all $m$-dimensional faces of $\Delta$, introduced in
  \cite[Definition 4.9]{Swa06} as a generalization of the short simplicial $h$-polynomial
  of $\Delta$.

  (c) Suppose that $P$ is finite of length $d$ and that every principal order ideal of
  $P$ is isomorphic to the poset of nonempty faces of a cube. Using Example
  \ref{ex:gammacube}, we find similarly that
    $$ h (P, x) \ = \ \sum_{t \in P} \ (2x)^{\rho(t)} \,
       (1-x)^{d - \rho(t)}. $$
  Thus we have $h (P, x) = h^{(sc)} (\KR, x)$ when $P = \fF(\KR)$ for some cubical
  complex $\KR$. More generally, if $m \in \NN$ is fixed and $P$ is the poset of faces
  of a pure cubical complex $\KR$ of dimension $m$ or higher, then $h (P, x)$ is equal to
  the sum of the $h$-polynomials of the links of all $m$-dimensional faces of $\KR$. \qed
  \end{example}

  We now recall the following key definition from \cite{Sta92}, extended to lower
  graded, locally Eulerian posets.

  \begin{definition} \label{def:formal} (\cite[Definition 7.4]{Sta92})
    Let $P$ be a lower graded, locally Eulerian poset. A \emph{formal subdivision} of $P$
    consists of a lower graded, locally Eulerian poset $Q$ and a surjective map
    $\sigma: Q \to P$ which have the following properties:
      \begin{itemize}
        \item[(i)] $\rho(t) \le \rho (\sigma(t))$ for every $t \in Q$.
        \item[(ii)] For every $u \in P$, the inverse image $Q_u := \sigma^{-1}
        (P_{\le u})$ is an order ideal of $Q$ of length $\rho(u)$, called the restriction
        of $Q$ to $u$.
        \item[(iii)] For every $u \in P$ we have
          \begin{equation} \label{eq:formal}
            x^{\rho(u)} \, h (Q_u, 1/x) \ = \ h (\inte(Q_u), x),
          \end{equation}
        where $\inte(Q_u) := \sigma^{-1} (u)$ is the interior of $Q_u$, the polynomial
        $h (Q_u, x)$ is defined by (\ref{eq:def-hgen}); i.e.,
          $$ x^{\rho(u)} \, h (Q_u, 1/x) \ = \ \sum_{t \in Q_u} \ \gamma_t (x) \,
            (x-1)^{\rho(u) - \rho(t)} $$
        and $h (\inte(Q_u), x)$ is defined by
          $$ x^{\rho(u)} \, h (\inte(Q_u), 1/x) \ = \
             \sum_{t \in \inte(Q_u)} \ \gamma_t (x) \, (x-1)^{\rho(u) - \rho(t)}, $$
        where $\gamma = \gamma(Q, \lambda): Q \to \KK[x]$.
      \end{itemize}
  \end{definition}

  We now summarize the main properties of $h$-vectors and local $h$-vectors of formal
  subdivisions. The proofs are straightforward generalizations of the proofs of Theorem
  7.5, Corollary 7.7 and Theorem 7.8 in \cite{Sta92} and will thus be omitted.

  \begin{theorem} \label{thm:gen}
    Let $\sigma: Q \to P$ be a formal subdivision of a lower graded, locally Eulerian
    poset $P$.
      \begin{itemize}
        \item[(a)] The function $f: P \to \KK[x]$ defined by $f_u (x) = h (Q_u, x)$ is
        $\lambda$-acceptable.
        \item[(b)] Assume that $P$ has a maximum element $\hat{1}$ and define
          \begin{equation} \label{eq:deflocal}
          \ell_P (Q, x) \ = \sum_{u \in P} \ h(Q_u, x) \, \xi^{-1}_{u \hat{1}} (x),
          \end{equation}
        where $\xi = \xi(P, \lambda) \in I(P)$. Then
          \begin{equation} \label{eq:genthmsymm}
            x^d \, \ell_P (Q, 1/x) \ = \ \ell_P (Q, x),
          \end{equation}
        where $d$ is the rank of $P$.
        \item[(c)] Assume that $P$ is graded of rank $d$. Then
          \begin{equation} \label{eq:genthmlocal}
          h (Q, x) \ = \sum_{u \in P} \ \ell_{P_{\le u}} (Q_u, x) \, h (P_{\ge u}, x),
          \end{equation}
        where $h$ is defined by (\ref{eq:def-hgen}).
      \end{itemize}
  \end{theorem}

  Parts (b) and (c) of Theorem \ref{thm:gen} generalize Theorems \ref{thm:symm} and
  \ref{thm:local}, respectively. Indeed, let $C$ be a $d$-dimensional cube and let $P
  = \fF(C)$. For each $u \in P$, the interval $[u, C]$ in $P$ is a Boolean lattice
  and hence $\xi_{st} (x) = 1$ for all $s \le_P t$. Thus, in view of Example
  \ref{ex:cchgen} (c), if $Q$ is a cubical subdivision $\Gamma$ of $C$, then the
  right-hand side of (\ref{eq:deflocal}) is equal to $\ell_C (\Gamma, x)$ and
  (\ref{eq:genthmsymm}) reduces to (\ref{eq:thmsymm}). For similar reasons,
  (\ref{eq:genthmlocal}) reduces to (\ref{eq:thmlocal}) in the case of cubical
  subdivisions of cubical complexes.

  \begin{remark} \label{rem:notaccept}
    For a cubical subdivision $\Gamma$ of a cube $C$, part (a) of Theorem \ref{thm:gen}
    asserts that
      $$ \sum_{F \, \le_P \, G} \ h^{(sc)} (\Gamma_F, x) \,(x-1)^{\dim(G) - \dim(F)}
      \ = \ x^{\dim(G)} \, h^{(sc)} (\Gamma_G, 1/x) $$
    holds for every $G \in P = \fF(C)$ (this indeed follows from (\ref{eq:mobinvG}) and
    Corollary \ref{cor:screc}). As the subdivision of the line segment with one interior
    vertex shows already, the corresponding property fails for the long cubical
    $h$-vector. \qed
  \end{remark}

  \begin{example} \label{ex:newsub}
  (a) Let $\Delta'$ be a simplicial subdivision of a simplicial complex $\Delta$, in
  the sense of \cite[Part I]{Sta92}, with associated subdivision map $\sigma: \Delta'
  \to \Delta$. The map $\sigma$ induces a surjective map $\sigma: Q \to P$, where $P$
  and $Q$ are the posets of nonempty faces of $\Delta$ and $\Delta'$, respectively.
  Endowed
  with this map, the poset $Q$ is a formal subdivision of $P$. Indeed, conditions (i)
  and (ii) of Definition \ref{def:formal} are clear. By Example \ref{ex:cchgen} (b),
  the polynomial $h(Q_F, x)$ is equal to the short simplicial $h$-polynomial of the
  restriction $\Delta'_F$ of $\Delta'$ to $F \in \Delta \sm \{\varnothing\}$. Thus,
  condition (iii) can be verified by an argument analogous to the one in the proof of
  Corollary \ref{cor:screc}. We note that if $\Delta$ (and hence $\Delta'$) is pure,
  then $h(Q, x)$ is equal to the short simplicial $h$-polynomial of $\Delta'$,
  discussed in Example \ref{ex:cchgen} (b).

  More generally, if $m \in \NN$ is fixed and $P$ and $Q$ are the posets of faces of
  $\Delta$ and $\Delta'$, respectively, of dimension $m$ or higher, then there is a
  formal subdivision $\sigma: Q \to P$ of $P$ in which $h(Q_F, x)$ is equal to the
  sum of the $h$-polynomials of the links of all $m$-dimensional faces of $\Delta'_F$,
  for every $F \in P$ (condition (iii) of Definition \ref{def:formal} can be verified
  again as in the proof of Corollary \ref{cor:screc}). Moreover, if $\Delta$ is pure,
  then $h (Q, x)$ is equal to the sum of the $h$-polynomials of the links of all
  $m$-dimensional faces of $\Delta'$.

  (b) Similar remarks hold for the induced map $\sigma: Q \to P$, when $P$ and $Q$ are
  the posets of faces of dimension $m$ or higher of two cubical complexes $\KR$ and
  $\KR'$, respectively, such that $\KR'$ is a cubical subdivision of $\KR$. \qed
  \end{example}

  \subsection{CW-regular subdivisions} \label{subsec:CW}
  We now consider a class of complexes which generalize regular CW-complexes. A
  triangulable space for us will be any topological space which is homeomorphic to the
  geometric realization of a finite simplicial complex. Let ${\rm C}$ be a finite
  collection of subspaces of a Hausdorff space $X$, called \emph{faces}, such that each
  element of ${\rm C}$ is homeomorphic either to the zero-dimensional ball or to a
  triangulable, connected manifold with nonempty boundary. Assume that the following
  conditions hold:
  \begin{itemize}
  \itemsep=0pt
  \item[$\bullet$] $\varnothing \in {\rm C}$,

  \item[$\bullet$] the interiors of the nonempty faces of ${\rm C}$ partition $X$ and

  \item[$\bullet$] the boundary of any face of ${\rm C}$ is a union of faces of ${\rm
  C}$.
  \end{itemize}
  For lack of better terminology, we will refer to such a complex ${\rm C}$ as a
  \emph{regular M-complex}. Thus, ${\rm C}$ will be a regular CW-complex if every face
  is homeomorphic to a ball. The intersection of any two faces of ${\rm C}$ is a union
  (possibly empty) of faces of ${\rm C}$. We say that ${\rm C}$ has the
  \emph{intersection property} if the intersection of any two faces of ${\rm C}$ is
  also a face of ${\rm C}$. We denote by $\fF({\rm C})$ the set of nonempty faces of
  ${\rm C}$, partially ordered by inclusion. This poset is lower graded but not
  necessarily locally Eulerian. Moreover, the homeomorphism type of $X$ is not
  necessarily determined by $\fF({\rm C})$. As was the case with cubical complexes,
  any topological properties of a regular CW-complex ${\rm C}$ we consider will refer
  to those of the order complex of $\fF({\rm C})$.

  Let ${\rm C}$ be a regular M-complex and set $P = \fF({\rm C})$. A (topological)
  \emph{CW-regular subdivision} of ${\rm C}$ is a regular CW-complex $\Gamma$,
  together with a (necessarily surjective) map $\sigma: \Gamma \sm \{\varnothing\}
  \to P$, satisfying conditions (i) and (ii) of Definition \ref{def:formal}, as well
  as the following, instead of condition (iii):
    \begin{itemize}
      \item[(iii$'$)] For every $u \in P$, the subcomplex $\Gamma_u := \sigma^{-1}
      (P_{\le u}) \cup \{\varnothing\}$ of $\Gamma$ is homeomorphic to the manifold $u$
      and $\inte(\Gamma_u) := \sigma^{-1} (u)$ is equal to the set of interior faces of
      this manifold.
    \end{itemize}
  This specializes to the notion of CW-regular subdivision of \cite[p. 839]{Sta92}
  when ${\rm C}$ is a regular CW-complex. The poset $Q = \fF(\Gamma)$ of nonempty
  faces of $\Gamma$ is lower graded and locally Eulerian and we have a surjective
  map $\sigma: Q \to P$ which satisfies conditions (i) and (ii) of Definition
  \ref{def:formal}. The following statement generalizes \cite[Proposition 7.6]{Sta92}.
  \begin{proposition} \label{prop:CWformal}
    Let ${\rm C}$ be a regular M-complex and assume that $P = \fF({\rm C})$ is locally
    Eulerian. Then every CW-regular subdivision of ${\rm C}$ is a formal subdivision.
  \end{proposition}
  \begin{proof}
    As in the special case of \cite[Proposition 7.6]{Sta92}, equation (\ref{eq:formal})
    can be verified by the computation in the proof of \cite[Lemma 6.2]{Sta87}. To
    facilitate the reader, we give the details as follows. We recall that $Q_u$ is the
    poset of nonempty faces of $\Gamma_u$ and that $\inte(Q_u) = \inte(\Gamma_u) =
    \sigma^{-1} (u)$ for $u \in P$ and set $\gamma = \gamma(Q, \lambda): Q \to \KK[x]$.
    For $u \in P$, we have
      \begin{eqnarray*}
      h (\inte(Q_u), x) &=& x^{\rho(u)} \sum_{t \in \inte(\Gamma_u)}
      \ \gamma_t (1/x) \, \left( \frac{1}{x} - 1 \right)^{\rho(u) - \rho(t)} \\
      & & \\
      &=& \sum_{t \in \inte(\Gamma_u)} \ x^{\rho(t)} \, \gamma_t (1/x) \,
      (1-x)^{\rho(u) - \rho(t)} \\
      & & \\
      &=& \sum_{t \in \inte(\Gamma_u)} \ (1-x)^{\rho(u) - \rho(t)} \
      \sum_{s \le_Q t} \ \gamma_s (x) \, (x-1)^{\rho(t) - \rho(s)} \\
      & & \\
      &=& \sum_{s \in \Gamma_u \sm \{\varnothing\}} \gamma_s (x) \,
      (x-1)^{\rho(u) - \rho(s)} \sum_{t \in \inte(\Gamma_u): \, t \ge_Q s}
      (-1)^{\rho(u) - \rho(t)}.
      \end{eqnarray*}
    Denoting the inner sum by $\nu(s)$ and noting that $Q$ is locally Eulerian,
    we find that
      \begin{eqnarray*}
      \nu(s) &=& \sum_{t \in \Gamma_u: \, t \ge_Q s} (-1)^{\rho(u) - \rho(t)}
      \ \, - \sum_{t \in \partial \Gamma_u: \, t \ge_Q s} (-1)^{\rho(u) - \rho(t)} \\
      & & \\
      &=& (-1)^{\rho(u) - \rho(s)} \, \left( - \mu_{\hat{Q}_u} (s, \hat{1}) +
      \mu_{\widehat{\partial Q}_u} (s, \hat{1}) \right),
      \end{eqnarray*}
    where $\partial Q_u$ is the poset of nonempty faces of $\partial \Gamma_u$ and
    $\hat{\Omega}$ denotes a poset $\Omega$ with a maximum element $\hat{1}$ adjoined.
    Since $\Gamma_u$ is a regular CW-decomposition of a manifold with nonempty boundary
    and $s$ is a nonempty face of $\Gamma_u$, the various M\"obius functions can be
    computed from \cite[Proposition 3.8.9]{StaEC1} and we find that $\nu(s) = 1$ in
    both cases $s \in \inte(\Gamma_u)$ and $s \in \partial \Gamma_u$. Hence
      $$ h (\inte(Q_u), x) \ = \sum_{s \in Q_u} \gamma_s (x) \,
      (x-1)^{\rho(u) - \rho(s)} \ = \ x^{\rho(u)} \, h (Q_u, 1/x) $$
    and the proof follows.
  \end{proof}

  \begin{remark} \label{rem:notintersect}
    A regular CW-complex $\Gamma$ will be called \emph{cubical} if each face of
    $\Gamma$ is combinatorially equivalent to a cube (such a complex may not have
    the intersection property). Example \ref{ex:cchgen} implies that the
    $h$-polynomial of such a
    complex is equal to its short cubical $h$-polynomial, as defined by the
    right-hand side of (\ref{eq:defsc}). Thus, we may deduce from Proposition
    \ref{prop:CWformal} and Theorem \ref{thm:gen} (c) that Theorem \ref{thm:local}
    continues to hold if $\KR$ is replaced by a pure cubical regular CW-complex
    ${\rm C}$ and $\KR'$ is replaced by a cubical CW-regular subdivision of ${\rm
    C}$. We should point out that the links of nonempty faces of ${\rm C}$ are no
    longer necessarily simplicial complexes. However, their face posets are
    simplicial, in the sense of \cite[p.~135]{StaEC1} \cite[p.~112]{StaCCA}. Thus 
    the $h$-vectors of these links, which appear in (\ref{thm:local}), should be 
    defined by (\ref{eq:hsp}). Similarly, Theorem \ref{thm:symm} continues to hold 
    if $\Gamma$ is replaced by a cubical CW-regular subdivision of a cube $C$. \qed
  \end{remark}

  \begin{example} \label{ex:cbs2}
    Let ${\rm C}$ be a regular CW-complex and set $P = \fF ({\rm C})$. Let $Q$ 
    denote the set of nonempty closed intervals in $P$, ordered by inclusion, and 
    consider the map $\sigma: Q \to P$ defined by $\sigma ([s, t]) = t$ for all $s 
    \le_P t$. One can easily check that the poset $Q$ is lower graded and locally 
    Eulerian (since so is $P$) and that $\sigma$ is surjective and satisfies 
    conditions (i) and (ii) of Definition \ref{def:formal}. 

    Under further assumptions (for instance, if the complex ${\rm C}$ is Boolean or 
    cubical), we have $Q = \fF(\Gamma)$ for some (cubical) regular CW-complex $\Gamma$ 
    and the resulting map $\sigma: \Gamma \sm \{\varnothing\} \to P$ is a CW-regular 
    subdivision of ${\rm C}$; see \cite[Section 2.3]{BBC97}. Let us compute $\ell_P 
    (Q, x)$ in one interesting special case, namely that in which ${\rm C}$ is the 
    simplex $2^V$ on a $d$-element vertex set $V$. Then $\Gamma$ is a cubical 
    complex with $d$ maximal faces, having a common vertex, and $\sigma$ is the 
    \emph{cubical barycentric subdivision}, studied in \cite{Het96}, of $2^V$. From 
    (\ref{eq:deflocal}) we have
      \begin{equation} \label{eq:localcbs2}
        \ell_P (Q, x) \ = \sum_{\varnothing \subset F \subseteq V} \ (-1)^{d - |F|}
        \, h^{(sc)} (\Gamma_F, x)
        \end{equation}
    where, setting $k = |F|$, 
      $$ h^{(sc)} (\Gamma_F, x) \ = \ \sum_{i=0}^{k-1} f_i (\Gamma_F) (2x)^i 
         (1-x)^{k-i-1}. $$
    Recall that $f_i (\Gamma_F)$ counts the number of pairs $s \subseteq t$ of nonempty 
    subsets of $F$ such that $|t \sm s| = i$. There are ${k \choose i}$ ways to choose 
    the elements of $t \sm s$ and, given any such choice, there are $2^{k-i} - 1$ ways 
    to choose those of $s$. We conclude that $f_i (\Gamma_F) = {k \choose i} (2^{k-i} - 
    1)$ and hence that 
      \begin{equation} \label{eq:shcbs2}
        h^{(sc)} (\Gamma_F, x) \ = \ \frac{2^k - (1+x)^k}{1-x}. 
      \end{equation}
    Using equations (\ref{eq:localcbs2}) and (\ref{eq:shcbs2}) and the binomial theorem, 
    we get 
      \begin{eqnarray*}
      \ell_P (Q, x) &=& \sum_{k=1}^d \ (-1)^{d - k} {d \choose k} \, \frac{2^k - (1+x)^k}
      {1-x} \\
      &=& 1 + x + x^2 + \cdots + x^{d-1}.
      \end{eqnarray*}
    Using the previous formula and (\ref{eq:genthmlocal}) and doing some more work, 
    one can express the (short or long) cubical $h$-vector of the cubical barycentric 
    subdivision of any simplicial (or Boolean) complex $\Delta$ in terms of the 
    simplicial $h$-vector of $\Delta$ and thus deduce the main results of \cite[Section 
    4]{Het96}. We leave the details to the interested reader.
  \qed
  \end{example}

  \begin{example} \label{ex:ann}
  Figure \ref{fig:ann} shows a graded, locally Eulerian poset of rank 2. This
  poset is isomorphic to the poset of nonempty faces of a two-dimensional regular
  M-complex ${\rm C}$ whose only two-dimensional face is homeomorphic to an annulus
  with boundary the union of two disjoint triangles. For the cubical CW-regular
  subdivision $\Gamma$ of ${\rm C}$, shown on the right of this figure, we compute
  that $h^{(sc)} (\Gamma, x) = 6x+6$ and $\ell_P (Q, x) = 6x$. \qed
  \end{example}

  \begin{figure}
  \epsfysize = 1.7 in \centerline{\epsffile{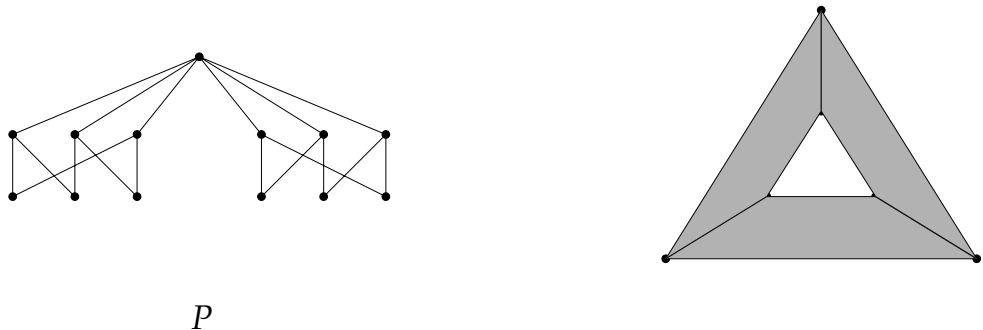}}
  \caption{Cubical subdivision of an annulus.}
  \label{fig:ann}
  \end{figure}

  It is an interesting problem to find conditions under which local $h$-vectors
  of CW-regular subdivisions are nonnegative. The proof of Theorem \ref{thm:nonn}
  works in the following situation. Let ${\rm C}$ be a regular M-complex with an
  inclusionwise maximum face, so that $P = \fF({\rm C})$ has a maximum element.
  The notion of a \emph{locally quasi-geometric} CW-regular subdivision of ${\rm C}$ 
  can be defined as for cubical subdivisions in Section \ref{sec:main}.

  \begin{theorem} \label{thm:nonnCW}
    Let $\sigma: \Gamma \sm \{\varnothing\} \to P$ be a locally quasi-geometric
    CW-regular subdivision of ${\rm C}$. If every closed interval in $Q =
    \fF(\Gamma)$ and $P = \fF({\rm C})$ is isomorphic to a Boolean lattice and
    $\Gamma$ has the intersection property, then $\ell_P (Q, x) \ge 0$.
  \end{theorem}

  \section*{Acknowledgments} The author wishes to thank Richard Stanley for 
  helpful comments and the anonymous referees for useful suggestions on the 
  presentation.


\begin{thebibliography}{xxx}
  %
  \bibitem{Ad96}
  R.M.~Adin,
  \textit{A new cubical $h$-vector},
  Discrete Math. {\bf~157} (1996), 3--14.
  %
  \bibitem{BBC97}
  E.K.~Babson, L.J.~Billera and C.S.~Chan,
  \textit{Neighborly cubical spheres and a cubical lower bound conjecture},
  Israel J. Math. {\bf~102} (1997), 297--315.
  %
  \bibitem{Bj84}
  A.~Bj\"orner,
  \textit{Posets, regular CW complexes, and Bruhat order},
  European J. Combin. {\bf~5} (1984), 7--16.
  %
  \bibitem{Bj95}
  A.~Bj\"orner,
  \textit{Topological methods}, in \textit{Handbook of Combinatorics}
  (R.L.~Graham, M.~Gr\"otschel and L.~Lov\'asz, eds.),
  North Holland, Amsterdam, 1995, pp.~1819--1872.
  %
  \bibitem{BSZ99}
  A.~Bj\"orner, M.~Las~Vergnas, B.~Sturmfels, N.~White and G.M.~Ziegler,
  Oriented Matroids,
  Encyclopedia of Mathematics and Its Applications {\bf~46}, Cambridge
  University Press, Cambridge, 1993; second printing, 1999.
  %
  \bibitem{EK07}
  R.~Ehrenborg and K.~Karu,
  \textit{Decomposition theorem for the cd-index of Gorenstein* posets},
  J. Algebraic Combin. {\bf~26} (2007), 225--251.
  %
  \bibitem{Gru67}
  B.~Gr\"unbaum,
  Convex Polytopes,
  Interscience, London, 1967; second edition, Graduate Texts in
  Mathematics {\bf~221}, Springer, New York, 2003.
  %
  \bibitem{HN02}
  P.~Hersh and I.~Novik,
  \textit{A short simplicial $h$-vector and the upper bound theorem},
  Discrete Comput. Geom. {\bf~28} (2002), 283--289.
  %
  \bibitem{Het95}
  G.~Hetyei,
  \textit{On the Stanley ring of a cubical complex},
  Discrete Comput. Geom. {\bf~14} (1995), 305--330.
  %
  \bibitem{Het96}
  G.~Hetyei,
  \textit{Invariants des complexes cubiques},
  Ann. Sci. Math. Qu\'ebec {\bf~20} (1996), 35--52.
  %
  \bibitem{Joc93}
  W.~Jockusch,
  \textit{The lower and upper bound problems for cubical polytopes},
  Discrete Comput. Geom. {\bf~9} (1993), 159--163.
  %
  \bibitem{Kle09}
  S.~Klee,
  \textit{Lower bounds for cubical pseudomanifolds},
  Discrete Comput. Geom. (to appear).
  %
  \bibitem{MW71}
  P.~McMullen and D.W.~Walkup,
  \textit{A generalized lower-bound conjecture for simplicial polytopes},
  Mathematika {\bf~18} (1971), 264--273.
  %
  \bibitem{Sav10}
  C.~Savvidou,
  \textit{Face numbers of cubical barycentric subdivisions},
  preprint, 2010, {\tt arXiv:1005.4156}.
  %
  \bibitem{StaEC1}
  R.P.~Stanley,
  Enumerative Combinatorics, vol.~1,
  Wadsworth \& Brooks/Cole, Pacific Grove, CA, 1986;
  second printing,
  Cambridge Studies in Advanced Mathematics {\bf~49},
  Cambridge University Press, Cambridge, 1997.
  %
  \bibitem{Sta87}
  R.P.~Stanley,
  \textit{Generalized $h$-vectors, intersection cohomology of toric
  varieties and related results},
  in \textit{Commutative Algebra and Combinatorics}
  (N.~Nagata and H.~Matsumura, eds.),
  Adv. Stud. Pure Math. {\bf~11}, Kinokuniya, Tokyo and North-Holland,
  Amsterdam-New York, 1987, pp.~187--213.
  %
  \bibitem{Sta92}
  R.P.~Stanley,
  \textit{Subdivisions and local $h$-vectors},
  J. Amer. Math. Soc. {\bf~5} (1992), 805--851.
  %
  \bibitem{Sta93}
  R.P.~Stanley,
  \textit{A monotonicity property of $h$-vectors and $h^*$-vectors},
  European J. Combin. {\bf~14} (1993), 251--258.
  %
  \bibitem{StaCCA}
  R.P.~Stanley,
  Combinatorics and Commutative Algebra,
  second edition, Birkh\"auser, Basel, 1996.
  %
  \bibitem{Swa06}
  E.~Swartz,
  \textit{$g$-elements, finite buildings and higher Cohen-Macaulay
  connectivity},
  J. Combin. Theory Series A {\bf~113} (2006), 1305--1320.
  %
  \bibitem{Zie95}
  G.M.~Ziegler,
  Lectures on Polytopes,
  Graduate Texts in Mathematics {\bf~152},
  Springer-Verlag, New York, 1995.
  %
  \end{thebibliography}
  \end{document}